\documentclass[12pt, oneside, reqno]{amsart}
\pdfoutput=1

\title[MDPG \& MFOSLS loss functions for neural networks]{
  Global error estimators for parametric monotone nonlinearities
and neural approximations}

\author{Pablo Cort{\'e}s Castillo}
\address{Portland State University, PO Box 751, Portland OR 97201}
\email{pcortes@pdx.edu}

\author{Wolfgang Dahmen}
\address{University of South Carolina, 1523 Greene Street, 29208 Columbia, South
  Carolina}
\email{dahmen@math.sc.edu}

\author{Jay Gopalakrishnan}
\address{Portland State University, PO Box 751, Portland OR 97201}
\email{gjay@pdx.edu}



\usepackage{calc}
\usepackage[margin=2.5cm]{geometry}
\usepackage[bbgreekl]{mathbbol}
\usepackage[cal=boondox]{mathalpha}
\usepackage{hyperref}
\usepackage{enumitem}
\usepackage{amssymb}
\usepackage{pgfplots}
\usepgfplotslibrary{groupplots}
\pgfplotsset{compat=1.17}

\newtheorem{theorem}{Theorem}[section]
\newtheorem{lemma}[theorem]{Lemma}

\newtheorem{proposition}[theorem]{Proposition}

\theoremstyle{definition}
\newtheorem{assumption}[theorem]{Assumption}

\newtheorem{definition}[theorem]{Definition}
\newtheorem{example}[theorem]{Example}

\theoremstyle{remark}
\newtheorem{remark}[theorem]{Remark}

\newcommand{\bb}[1]{\mathbb{{#1}}}
\newcommand{\cl}[1]{\mathcal{{#1}}}
\let\div\relax
\DeclareMathOperator{\tr}{{\ensuremath{\text{tr}}}}
\newcommand{\trn}{\mathop{\mathrm{tr}_n}}
\DeclareMathOperator{\osc}{{\ensuremath{\text{osc}}}}

\DeclareMathOperator{\div}{{\ensuremath{\text{div}}}}

\DeclareMathOperator*{\argmin}{arg\,min}

\DeclareMathOperator*{\dive}{div}
\DeclareMathOperator*{\ran}{range}

\DeclareFontFamily{U}{mathx}{\hyphenchar\font45}
\DeclareFontShape{U}{mathx}{m}{n}{
<-6> mathx5 <6-7> mathx6 <7-8> matha7
<8-9> mathx8 <9-10> mathx9
<10-12> mathx10 <12-> mathx12
}{}
\DeclareSymbolFont{mathx}{U}{mathx}{m}{n}
\DeclareMathSymbol{\bigplus}{\mathop}{mathx}{"90}
\DeclareMathSymbol{\bigtimes}{\mathop}{mathx}{"91}

\newcommand{\veps}{\varepsilon}

\newcommand{\ip}[1]{\langle{#1}\rangle}
\newcommand{\fp}{\varPi^a}
\newcommand{\fpb}{\varPi^b}
\newcommand{\PB}{P_\alpha} %
\newcommand{\Ba}{\cl{B}_\alpha}
\newcommand{\Aa}{\cl{A}_\alpha}
\newcommand{\Ca}{\cl{C}_\alpha}
\newcommand{\ba}{b_\alpha}

\newcommand{\Rza}{\cl{G}_\alpha}

\newcommand{\La}[1]{\hat{\cl{L}}(\alpha, {#1})}
\newcommand{\Ea}[2]{\cl{e}_h^{#1}(\alpha, {#2})}

\newcommand{\qt}{\tilde{q}}

\newcommand{\pt}{\tilde{p}}

\newcommand{\clu}{{\cl{u}}}
\newcommand{\clv}{{\cl{v}}}
\newcommand{\clvh}{\clv_h}

\newcommand{\clw}{{\cl{w}}}
\newcommand{\clwh}{{\clw_h}}
\newcommand{\clz}{{\cl{z}}}
\newcommand{\clzh}{{\clz_h}}

\newcommand{\Fh}{{\cl{F}_h}}
\newcommand{\Fht}{{\cl{F}_{h,\theta}}}

\newcommand{\Ua}{{\cl{U}_\alpha}}
\newcommand{\Uap}{{\cl{U}'_\alpha}}
\newcommand{\Uah}{{\cl{U}_{\alpha, h}}}

\newcommand{\Va}{{\cl{V}_\alpha}}
\newcommand{\Vap}{{\cl{V}'_\alpha}}
\newcommand{\Vah}{\cl{V}_{\alpha, h}}
\newcommand{\VK}{\cl{V}(K)}
\newcommand{\VhK}{\cl{V}_h(K)}
\newcommand{\cQ}{\cl{Q}}

\newcommand{\om}{\varOmega}
\newcommand{\oh}{{\varOmega_h}}
\newcommand{\doh}{{\partial\varOmega_h}}

\newcommand{\RT}{RT}

\newcommand{\Ho}{\mathring{H}}

\newcommand{\Hdivoh}{H(\dive, \oh)}
\newcommand{\Hoh}{\Ho^{1/2}(\partial\oh)}
\newcommand{\Hmh}{H^{-1/2}(\partial\oh)}
\newcommand{\Hdiv}{H(\dive)}
\newcommand{\Hdzero}{H(\text{div0})}

\renewcommand{\Omega}{\om}

\usepackage[font=footnotesize]{caption} %

\renewcommand{\fpb}{\varPi}
\renewcommand{\fp}{\varPi}
\newcommand{\Ma}[1]{\cl{M}_\alpha({#1})}

\begin{document}
\begin{abstract}
  We construct computable error estimators, which double as
  loss functions for neural networks,  for a
  class of parametric nonlinear partial differential equations with a
  monotonicity property, and prove that they are  globally reliable and
  efficient. The value of such a loss function
  is bounded above and below by the squared
  error in the natural trial norm, for every trial function, not
  merely for those near the exact solution; this global property rests
  on monotonicity.
  The construction rests on
  splitting the nonlinear operator into a linear part and a strongly
  monotone closure, and on measuring the linear residual in a discrete
  dual norm.  Since the closure contributes a dual norm that admits no
  closed form when the trial norm is stronger than an $L_2$ norm, the
  estimator is built around a computable surrogate for it, required
  only to satisfy a pairing bound and a Lipschitz bound.  The main
  theorem then yields two-sided bounds with explicit constants and
  covers two instances. A first-order system least-squares estimator
  on conforming trial spaces is the first instance studied: its
  two-sided bound holds on the whole trial space and therefore applies
  to arbitrary approximations, including those not in a discrete
  finite element space. The second instance is a discontinuous
  Petrov--Galerkin estimator on trial spaces of finite element
  functions for each parameter value, built with broken test spaces,
  for which no conformity is required and the dual norm is computed by
  independent element-local problems.  All assumptions are verified
  for a model class of nonlinear fluxes, with the constants tracked
  explicitly in terms of the parameter range.  Being computable for an
  arbitrary input, both estimators serve as variationally correct loss
  functions for neural network approximations of parameter-to-solution
  maps.
\end{abstract}
\maketitle

\section{Introduction}
\label{sec:intro}

This paper develops a computable, globally reliable and efficient {\em
  a posteriori} error estimator for a class of nonlinear diffusion
problems, in two guises: a discontinuous Petrov--Galerkin (DPG)
estimator built on broken test spaces, and a first-order system
least-squares (FOSLS) estimator built on conforming ones.  Both are
instances of a single abstract result.  Because the estimator is a
computable functional of an arbitrary trial function (not merely of a
discrete solution), it doubles---in the context of operator learning---as a neural network {\em loss function}
whose value certifies accuracy. More precisely, we apply our findings to learning the parameter-to-solution map for parameter-dependent nonlinear diffusion problems,
taken up in~\S\ref{sec:nn} and further commented on below. The estimator theory stands on its own,
though,
and is of interest wherever two-sided error control for monotone
nonlinear partial differential equations (PDEs) is needed, adaptivity
being the obvious example.

The difficulty we must overcome is intrinsic to residual-based error
control for nonlinear parametric operator equations $\Aa(\clu) = \ell$
seeking $\clu$ in some normed space $\Ua$.  (Throughout $\alpha$ is a
high-dimensional parameter within the PDE.) What one wants is a computable error
estimator for any candidate $\clw$, denoted by $\cl L(\alpha, \clw)$,
obeying
\[
  c \, \| \clu - \clw \|_{\Ua}^2
  \;\le\; \cl L(\alpha, \clw)
  \;\le\; C \, \| \clu - \clw \|_{\Ua}^2
\]
with constants independent of the
discretization---the property called {\em variational correctness} in~\cite{BDO}.
For linear well-posed problems such two-sided bounds follow from 
the Babu\v{s}ka--Ne\v{c}as theorem, and the residual in the dual
norm $\| \ell - \Aa \clw \|_{\Vap}$ is the natural candidate for
$\cl L(\alpha, \clw).$  Two
obstructions arise in the nonlinear case.  First, existence and
uniqueness require structural hypotheses on the nonlinearity
(such as local derivative control near the solution or monotonicity
\cite{Showa96,Zeidl90}), and the stability that underlies a two-sided
bound is then typically only {\em local}, valid in a neighborhood of
a regular solution at which the  derivative is an
isomorphism~\cite{carstensen2017}.  Second, and already present in the
linear case, the dual norm is not computable: evaluating
$\| \cdot \|_{\Vap}$ exactly requires solving a global problem in
$\Va$.

We resolve both by combining two ideas in Cantin and Heuer~\cite{CH}.  The first is a splitting of
the nonlinear operator into a
{\em linear} operator $\Ba$ and a  nonlinear
{\em closure} $\Ca$
that is strongly monotone on the kernel of $\Ba$.
The residual then separates into a linear part,
to which the well-developed linear theory applies, and a closure part
controlled by monotonicity; the resulting bounds are {\em global},
holding for every trial function rather than only near the solution.
The second is the DPG device of breaking the test space
\cite{DemkoGopal11,DPGacta}: with an element-local test space, the
dual norm of the linear residual decouples into independent local
problems and becomes computable, at the price of hybrid trial spaces
carrying interface traces, and of a Fortin operator~\cite{GopalQiu14}
relating the broken discrete dual norm to the continuous one.

Two aspects of the present study  go beyond~\cite{CH}. First, while~\cite{CH} only provides {\em a priori} error analysis, we provide an error estimator and an {\em a posteriori} analysis.
Second, we also develop a least-squares type error estimator for the same nonlinear PDE. In doing so we encounter a third obstruction  and it
dictates the form of our abstract theory (which covers both DPG
and least-squares instances).
The nonlinear closure contributes a term $\| \Ca(\clw) \|_{\Uap}$ which is
itself a dual norm, and which admits a closed form only when the trial
norm is of $L_2$ type.  In the least-squares setting the trial norm is
a graph norm,
strictly stronger than the $L_2$ product norm, and the $\Uap$-norm is not locally computable.  We therefore do not treat the exact dual
norm as the object of study.  Instead, the estimator is built from the
outset around a computable surrogate $\cl M_\alpha$ that 
satisfies two abstract conditions. A single theorem then delivers global reliability and
efficiency, with constants exhibited explicitly in terms of the
monotonicity and Lipschitz data of the closure and the norm of the
Fortin operator.

Two instances of this theorem occupy the rest of the paper.  In the least-squares instance, named 
{\em MFOSLS} loss function,  the trial space is conforming
and the two-sided bound holds on the entire
trial space $\Ua$, so that the estimator controls the error of an
{\em arbitrary} approximation, of finite element type or not.  In the DPG instance, named 
{\em MDPG} loss, the formulation is ultraweak and hybrid, the test
space is broken, and a Fortin operator and an oscillation term enter,
the latter vanishing for piecewise constant data.  The two instances
share one theorem but differ in what they demand of the approximant:
MFOSLS requires conformity---the flux $q$ must be in $ \Hdiv$ and the primal solution variable $u$ must be in $ \Ho^1(\om)$---
while MDPG places all interior variables in $L_2$.  We verify every assumption, for a class of
nonlinear fluxes covering standard models, and track how the constants
depend on the parameter range.

Our development is related to several strands of work.  Nonlinear DPG,
as residual minimization over a broken test space, has been formulated in~\cite{carstensen2017}, where
a priori quasi-optimality and a posteriori bounds were obtained under
 local stability. This generalized the a posteriori error control for
linear DPG methods developed earlier in~\cite{CarstDemkoGopal14}.  The
least-squares side has its own long lineage
\cite{FOSLS_coercivity,BocheGunzb09a}, in which coercivity of the
first-order system furnishes the error--residual equivalence directly.
Residual minimization in dual norms for nonlinear problems has
attracted renewed attention recently, notably in Banach spaces with
the $p$-Laplacian as model in~\cite{MugaPererRojasRuiz26}. Closer to
the present viewpoint
are variationally correct losses for linear parametric problems
studied in~\cite{BDO,CDG} and~\cite{FuhrRojas25}, where the error of an
arbitrary approximation of a linear well-posed problem is shown
equivalent to a computable residual plus an estimable remainder.
Another  independent contemporaneous work~\cite{VeraHeuerPraet2026} 
employs the auxiliary-variable device of~\cite{CH}, but for the
narrower class of semilinear problems in which the nonlinearity enters
only through additive lower order terms, not through the leading order
flux.  In place of the monotonicity used in~\cite{CH}, they enforce
the constitutive relations by an $L_2$ least-squares penalty under a
Lipschitz assumption, and prove a C\'ea-type quasioptimality estimate,
conditional on the existence of an exact solution and on an 
Aubin--Nitsche-type closeness hypothesis.  The estimator developed
below, in contrast, relies on monotonicity so is unconditionally reliable and efficient and covers nonlinear fluxes.

We apply these results in the currently vibrant area of operator
learning; see, e.g., \cite{Stu2,KLLABSA2023,BHKS2021,DDP}. One distinguishes
purely data-driven methods, to which most of these references belong, from
so-called ``physics-informed'' approaches. The former require a large number
of high-fidelity solutions as training data in order to perform regression,
usually in an $L_2$-type space. The latter avoid this offline cost by
training with residual-type loss functions; see, e.g.,
\cite{Kar1,Kar2,Kar3,Mishra1,Pardols,ZeinhMasriMarda24}. Since the uncertain
success of the optimization renders a priori expressivity results of little
help for assessing a training outcome, the only available source of
information is the loss function itself.
For linear PDEs, this observation has led in~\cite{BDO} to the notion of
variational correctness, requiring a parameter-uniform proportionality
between error and residual in norms dictated by a stable variational
formulation. As noted above, this may require dual-norm residuals. An
important first step towards practicality, already exploited
in~\cite{BDO} (see also~\cite{li2024}, without accuracy quantification), is
to reformulate the PDE as a first-order system. For a class of relevant
models, $L_2$-least-squares formulations of such systems are stable, so that
the dual-norm problem disappears;
see~\cite{FOSLS_coercivity,BocheGunzb09a,CaiPINN} and, for time-dependent
problems,~\cite{KF1,KF2,GS}. Combined with POD model reduction, this has
been used in~\cite{QCD} to learn accuracy-controlled parameter-to-solution
maps for second-order elliptic problems with mixed boundary conditions and
for linear elasticity. The alternative route via ultraweak DPG formulations
was compared with FOSLS for linear problems in~\cite{CDG}, with respect to
parameter robustness. The present paper pursues both routes for monotone
nonlinear diffusion, in either case with variational correctness.

The paper is organized as follows.  Section~\ref{sec:MDPG} sets up the
abstract framework, states the assumptions on the splitting, on the
Fortin operator and on the surrogate $\cl M_\alpha$, defines the error
estimator and proves its global reliability and efficiency; the
least-squares and DPG instances of the abstract setting are described
there.  Section~\ref{sec:nonlinear-diffusion} specializes to
parametric diffusion with nonlinear fluxes and exhibits the two
concrete estimators.  Section~\ref{sec:verification} verifies all
assumptions for a model class of fluxes, with explicit constants.
Section~\ref{sec:nn} describes the resulting loss functions for
learning parameter-to-solution maps, and Section~\ref{sec:numerics} closes
with numerical experiments comparing the MDPG and MFOSLS losses as training
objectives for neural network surrogates of the parameter-to-solution map:
the calibration of the loss against the error across parameter ranges, the
effect of network depth, training-set size, and polynomial degree, and an
example with sixteen parameters.

\section{Monotone loss functions}
\label{sec:MDPG}

Let $\Ua$ and $\Va$ be Hilbert spaces over $\bb R$, depending on a
parameter $\alpha$ in a compact set $\cQ$ and let $\cl f \in \Vap,$ the
dual space of $\Va$. We study the parametric family of nonlinear
partial differential equations cast as an  operator equation
\begin{equation}
  \label{eq:Au=ell}
  \Aa(\clu) = \cl f
\end{equation}
for a nonlinear operator $\Aa : \Ua \to \Vap$. When a unique solution
to~\eqref{eq:Au=ell} exists, the parameter-to-solution map
$\alpha \mapsto \clu(\alpha)$ is well defined. Our objective is
to approximate this map.

When the PDE nonlinearity is sufficiently structured to admit a
reformulation exploiting monotonicity, 
the framework of~\cite{CH} considers nonlinear problems
posed
in terms of two operators, a bounded {\em linear} operator
$\Ba : \Ua \to \Vap$ and a continuous operator
$\Ca : \Ua \to \Uap$ to which all the nonlinearity is confined, and
which is strongly {\em monotone} in the directions of the null space
of~$\Ba$ (see~\eqref{eq:S4-mono} below). Specifically, the nonlinear problem
is to  find a $\cl u \in \Ua$ satisfying
\begin{equation}
  \label{eq:A-split-BandC}
  \begin{pmatrix}
    \Ba \clu \\ \Ca(\clu) 
  \end{pmatrix}
  =
  \begin{pmatrix}
    \ell \\ 0 
  \end{pmatrix}
  \quad \text{ in } \Vap \times \Uap.
\end{equation}
Comparing this system with \eqref{eq:Au=ell}, the left- and right-hand 
sides define the specific structures of the operator $\Aa(\clu)$ and 
the functional $\cl f$. While some problems naturally present themselves 
in the standard form of \eqref{eq:Au=ell}, several interesting 
applications emerge when \eqref{eq:Au=ell} is reformulated into the 
split system \eqref{eq:A-split-BandC}. As we will see in a class of 
examples later in~\S\ref{sec:mdpg-loss-function-diffus}, this is done 
by extending the trial space in~\eqref{eq:Au=ell} to include an extra
variable. Namely, a nonlinear constitutive relation defining the extra
variable is enforced weakly by $\Ca$, referred to as a {\em nonlinear
  closure}. The remaining equations, rendered linear by the
extra variable, are collected into~$\Ba$. 

The main purpose of this section is to extend the results of~\cite{CH}
to provide a new globally reliable and efficient error estimator (in
Theorem~\ref{thm:gen-nonlinear-dpg} below). It then generates
variationally correct loss functions.
Two instances, which we refer to as the ``MFOSLS loss function'' and 
the ``MDPG loss function,'' are developed in the remainder of the paper.  
The theoretical analysis of this section is based on the assumptions we 
now describe.  Let
$\ker(\Ba) := \{ \clz \in \Ua : \Ba \clz = 0 \}$ and let
$\PB: \Ua \to \ker(\Ba)$ be the $\Ua$-orthogonal projection onto
$\ker(\Ba)$.  The bilinear form
generated by~$\Ba$ is defined by $\ba(\clz, \clv) := (\Ba \clz)(\clv)$
for all $\clz \in \Ua, \clv \in \Va$. For any normed space $X$, the
duality pairing of a functional $f \in X'$ and $z \in X$ is denoted by $f(z)$ or 
$\ip{f, z}_X$ and the subscript $X$ is omitted when understood from
context.

\subsection{The error estimator under general assumptions}

Motivated by the theory in~\cite{CH}, we isolate the following assumption
on the undiscretized problem setting.

\begin{assumption}
  \label{asm:MDPG}
  We assume that there exist positive constants $c_B$, $c_U$, $c_V$, and
  $c_{\rm Lip}$ such that for all $\clz, \clz_1, \clz_2 \in \Ua$,
  the operators $\Ba$ and $\Ca$ of the
  reformulation~\eqref{eq:A-split-BandC} satisfy
  \begin{align}
    \label{eq:S4-B-surjective}
    &\ran(\Ba) = \Vap, \text{ i.e., $\Ba$ is surjective,}
    \\
    \label{eq:S4-infsup}
    &\| \clz - \PB \clz \|_{\Ua}^2 \le c_B \| \Ba \clz \|_{\Vap}^2,
    \\
    \label{eq:S4-mono}
    &c_U \| \PB(\clz_1 - \clz_2) \|_{\Ua}^2
    \le \langle \Ca(\clz_1) - \Ca(\clz_2), \clz_1 - \clz_2 \rangle
    + c_V \| \Ba(\clz_1 - \clz_2) \|_{\Vap}^2,
    \\
    \label{eq:C-zero}
    &\langle \Ca(\clz), \clz_0 \rangle = 0  \text{  for all }
    \clz_0 \in \ker(\Ba) \implies
    \Ca(\clz) = 0 \text{ in } \Uap,      \text { and } 
    \\
    \label{eq:S4-Lip}
    &\| \Ca(\clz_1) - \Ca(\clz_2) \|_{\Uap}
    \le c_{\rm Lip} \| \clz_1 - \clz_2 \|_{\Ua}.
  \end{align}
\end{assumption}
Note that \eqref{eq:S4-infsup} says that $\Ba$ is lower-bounded on 
$\Ua/{\rm ker}(\Ba)$. To create practically computable loss functions, we 
shall  use a subset $\Uah \subseteq \Ua$ of trial functions (such as finite 
element functions, or realizations of neural networks) and a closed subspace
$\Vah \subseteq \Va$ whose dual will be used to measure residuals. Choosing 
a strict subspace $\Vah\subset\Va$ is relevant when the dual norm 
$\|\cdot\|_{\Va'}$ is not computable. The next assumption addresses this 
situation through a standard Fortin condition that ties the choice of $\Vah$
to the trial space $\Uah$.

\begin{assumption}
  \label{asm:fortin-b}
  We assume that there is a continuous linear operator
  $\fpb: \Va \to \Vah$
  satisfying
  \begin{equation}
    \label{eq:Fortin-b}
    \ba(\clw_h, \clv - \fpb \clv) = 0,
    \quad \text{ for all } \clwh \in \Uah, \; \clv \in \Va.
  \end{equation}
\end{assumption}

The first assumption, Assumption~\ref{asm:MDPG}, concerning the
undiscretized setting, consists of properties extracted from the
framework of~\cite{CH}.  Although the assumption presents many
conditions to verify, the global well-posedness (given below) is
guaranteed by the conditions, due to the results in~\cite{CH}.  Note
that Assumption~\ref{asm:fortin-b} imposes a standard Fortin condition
on the {\em linear} operator $\Ba$.  The existence of such Fortin
operators has already been verified for many common  linear 
operators~\cite{GopalQiu14} in the DPG setting discussed in 
\S\ref{ssec:mdpg-abstract}.  Note also that the choice $\Vah = \Va$
together with the identity operator $\fpb = \mathrm{id}$ always
satisfies Assumption~\ref{asm:fortin-b}. This trivial choice is of
practical relevance when the test space $\Va$ is self-dual (an
$L_2$-type space) and leads to the least-squares loss function
discussed in \S\ref{ssec:mfosls-abstract}.

The ``monotone DPG mixed formulation'' of~\cite{CH} uses a
stabilization parameter $\kappa > 0$ and seeks
$(\clv, \clu) \in \Va \times \Ua$ satisfying
\begin{equation}\label{eq:mixedC&H}
    \begin{pmatrix}
        -\kappa^{-1} \Rza & \Ba \\
        \Ba'& \Ca
    \end{pmatrix}
    \begin{pmatrix}
        \cl v \\ \cl u
    \end{pmatrix}
    =
    \begin{pmatrix}
        \ell \\ 0
    \end{pmatrix}
    \qquad
    \text{ in } \; \Vap \times \Uap,
\end{equation}
where $\Rza : \Va \to \Vap$ is the Riesz isometry,
and $\Ba': \Va \to \Uap$ is the  dual
of~$\Ba$. Eliminating $\clv=\kappa \Rza^{-1}(\Ba \clu-\ell)$ from the above system
\eqref{eq:mixedC&H},   we obtain a reduced form of the problem: find
$\clu \in \Ua$ such that
\begin{equation}
  \label{eq:SchurComp-problem}
    \langle \cl S_{\kappa} \clu, \clw \rangle
    = \kappa \langle \Rza^{-1} \ell, \Ba \clw \rangle,
    \qquad \text{ for all } \clw \in \Ua,
\end{equation}
where the Schur complement $\cl S_\kappa$ is given by
\begin{equation}
  \label{eq:SchurComp}
    \langle \cl S_\kappa \clz, \clw \rangle
    := \kappa \langle \Rza^{-1} \Ba \clz, \Ba \clw \rangle
    + \langle \Ca(\clz), \clw \rangle,
    \qquad \clz, \clw \in \Ua.
\end{equation}
Note that unlike standard DPG mixed formulations, there is no variable like
$\veps$ in~\eqref{eq:mixedC&H} which represents the full residual of
the original problem~\eqref{eq:A-split-BandC} that could serve as an error 
estimator. Below, we extend the analysis of~\cite{CH} to nevertheless provide 
also in the present context a globally reliable and efficient error estimator.

First, let us recall the salient results proved in~\cite{CH}.  The
penalization parameter $\kappa$ is fixed once and for all so that
\begin{equation}
  \label{eq:S4-kappa}
  \kappa > \kappa_0 := c_V + c_B c_U.
\end{equation}
By \cite[Proposition~2.1 and Theorem~2.3]{CH},
conditions~\eqref{eq:S4-infsup} and~\eqref{eq:S4-mono} of
Assumption~\ref{asm:MDPG}, together with this choice of $\kappa$
(the remaining coercivity hypothesis of \cite[Proposition~2.1]{CH}
holds with constant one since $\Rza$ is the Riesz isometry),
guarantee that~\eqref{eq:mixedC&H} is uniquely solvable and that 
\begin{equation}
  \label{eq:S4-key}
  c_U \| \clz_1 - \clz_2 \|_{\Ua}^2
  \le \langle \Ca(\clz_1) - \Ca(\clz_2), \clz_1 - \clz_2 \rangle
  + \kappa \| \Ba(\clz_1 - \clz_2) \|_{\Vap}^2
\end{equation}
holds for all $\clz_1, \clz_2 \in \Ua$.  This estimate shows the {\em
  strong monotonicity} property (as defined, say, in~\cite{Showa96})
of the Schur complement.  Moreover, by~\cite[Proposition~2.4]{CH},
conditions~\eqref{eq:S4-B-surjective} and~\eqref{eq:C-zero} of
Assumption~\ref{asm:MDPG} imply that the component $\clu$ of the
unique solution $(\clv, \clu) \in \Va \times \Ua$
of~\eqref{eq:mixedC&H} satisfies~\eqref{eq:A-split-BandC}.  Indeed,
the kernel condition~\eqref{eq:C-zero} first yields $\Ca(\clu) = 0$ from the second equation in 
\eqref{eq:mixedC&H}.
Hence $\Ba'\clv=0$, which by injectivity of $\Ba'$ (see \eqref{eq:S4-B-surjective}), implies $\clv=0$ 
and therefore $\Ba \clu=\ell$.

The design of our error estimator is motivated by the residual of the
two-equation system in~\eqref{eq:A-split-BandC}, $\Ba \clu = \ell$ and
$\Ca(\clu)=0$. This suggests that the residual of some $\clz \in \Ua$
should consist of terms involving $\| \ell - \Ba \clz \|_{\Vap}$ as
well as $\| \Ca \clz \|_{\Uap}$. These two contributions must be scaled
differently, as the analysis of the Schur complement
problem~\eqref{eq:SchurComp-problem} shows. Indeed, at a
$\clz \in \Ua$, the residual functional
of~\eqref{eq:SchurComp-problem} is
$\rho(\clw)= \ip{\cl S_\kappa \clz, \clw} - \kappa \langle \Rza^{-1}
\ell, \Ba \clw \rangle,$ which we can rewrite as
$ \rho(\clw) = \kappa \langle \Rza^{-1}(\Ba \clz - \ell), \Ba \clw
\rangle + \langle \Ca(\clz), \clw \rangle.  $ This motivates the
definition of the following loss function, or residual error estimator,
\begin{equation}\label{eq:C&H-loss}
  \La \clz = \kappa \| \ell - \Ba \clz \|_{\Vap}^2
  + \| \Ca(\clz) \|_{\Uap}^2,
  \qquad \clz \in \Ua, \; \alpha \in \cQ.
\end{equation}
However, $\La\cdot$ is defined in an infinite-dimensional setting and
neither of its two terms need be computable in general.  For the first
term, we replace the $\Vap$-norm by the $(\Vah)'$-norm, whose
computability is governed by the choice of a subspace
$\Vah \subseteq \Va$, as discussed in the instances
(\S\ref{ssec:mfosls-abstract} and \S\ref{ssec:mdpg-abstract}) below.
For the second term, we admit computable surrogates $\Ma\cdot$ of the
dual norm $\| \Ca(\cdot) \|_{\Uap}$, subject to the following
assumption.

\begin{assumption}
  \label{asm:N}
  We assume that $\Ma\cdot : \Ua \to [0, \infty)$ is a functional
  satisfying, for some constant $c_M > 0$,
  \begin{align}
    \label{eq:N-pairing}
    \ip{ \Ca(\clz), \clw }  \le\;
    & \Ma{\clz} \, \| \clw \|_{\Ua},
    \\
    \label{eq:N-lip}
    & \Ma{\clz}   \le c_M \| \clz - \bar{\clz} \|_{\Ua}
      \quad \text{whenever } \bar\clz \in \Ua \text{ satisfies }
      \Ca(\bar\clz) = 0,
  \end{align}
  for all $\clz, \clw \in \Ua$.
\end{assumption}
The trivial choice $\Ma{\clz} := \| \Ca(\clz) \|_{\Uap}$,
of course, satisfies \eqref{eq:N-pairing} of the assumption by the
definition of the dual norm, and~\eqref{eq:N-lip} with
$c_M = c_{\rm Lip}$ by~\eqref{eq:S4-Lip}, but it need not be
computable.

\begin{definition}
  \label{defn:mdpg}
  Given a closed subspace $\Vah \subseteq \Va$ and a functional
  $\Ma\cdot$ satisfying Assumption~\ref{asm:N}, the generated
  {\bf{error estimator}} or 
  {\bf{loss function}} is defined by
  \begin{equation}
    \label{eq:C&H-lossh}
    \cl L(\alpha, \clz_h)
    := \kappa \| \ell - \Ba \clz_h \|_{(\Vah)'}^2
    + \Ma{\clz_h}^2,
    \qquad \clzh \in \Uah, \; \alpha \in \cQ.
  \end{equation}
\end{definition}

In the next theorem, the terms 
\[
  \osc(\ell) = \sup_{0 \ne \clv \in \Va}
  \frac{|\ell(\clv - \fpb\,\clv)|}{\|\clv\|_{\Va}}
  \quad\text { and } \quad 
  \| \fp  \|_{L(\Va)} = \sup_{0 \ne \clv \in \Va}
  \frac{ \|\fp  \clv \|_{\Va}}{\| \clv \|_{\Va}}.
\]
represent a data oscillation term and the operator norm of $\fpb$, respectively.
The operator norm of $\Ba$ is denoted by $\| \Ba \|_{L(\Ua, \Vap)}$.

\begin{theorem}[Global variational correctness]
  \label{thm:gen-nonlinear-dpg}
  Suppose Assumptions~\ref{asm:MDPG}, \ref{asm:fortin-b},
  and~\ref{asm:N} hold and
  $\kappa$ satisfies~\eqref{eq:S4-kappa}. Then \eqref{eq:mixedC&H} has
  a unique solution whose $\clu$-component solves $\Ba \clu = \ell$
  and $\Ca(\clu) = 0$.  Moreover, the loss function $\cl L$
  of Definition~\ref{defn:mdpg} is
  variationally correct, i.e., there are positive constants $c$ and
  $C$ such that for any $\alpha \in \cQ$ and $\clwh \in \Uah$,
  \begin{align}
    \label{eq:reliability-L}
    c & \| \clu  - \clw_h \|_\Ua^2  \le
    \cl L(\alpha,\clw_h)\,
          \| \fpb \|_{L(\Va)}^2
          + \osc(\ell)^2,
    \\
    \label{eq:efficiency-L}
    \cl L(\alpha,\clw_h) \le C & \| \clu - \clw_h \|_\Ua^2.
  \end{align}
  We may choose $C = \kappa \| \Ba \|_{L(\Ua, \Vap)}^2 + c_M^2$,
  with $c_M$ from \eqref{eq:N-lip}, and $c$, defined
  in~\eqref{eq:8} below, is determined by the $c_U$ from~\eqref{eq:S4-mono}.
\end{theorem}
\begin{proof}
  Conditions~\eqref{eq:S4-B-surjective} and~\eqref{eq:C-zero} of
  Assumption~\ref{asm:MDPG} imply that \eqref{eq:mixedC&H} has a
  unique solution and component $\clu$ of the unique solution
  $(\clv, \clu) \in \Va \times \Ua$ of~\eqref{eq:mixedC&H}
  satisfies~\eqref{eq:A-split-BandC} (see~\cite[Theorem~2.3 and
  Proposition~2.4]{CH}). We now focus on proving the new 
  estimates~\eqref{eq:reliability-L} and~\eqref{eq:efficiency-L}.

  {\em Step 1.}  Let us first prove the reliability
  estimate~\eqref{eq:reliability-L}. Since $\Ba \clu = \ell$ and
  $\Ca(\clu) = 0$, applying~\eqref{eq:S4-key} with $\clz_1 = \clu$ and
  $\clz_2 = \clwh$,
  \begin{equation}
    \label{eq:rel-basic}
    c_U \| \clu - \clwh \|_{\Ua}^2
    \le \langle \Ca(\clu) - \Ca(\clwh), \clu - \clwh \rangle
    + \kappa \| \Ba(\clu - \clwh) \|_{\Vap}^2.
  \end{equation}
  We proceed to bound the two terms on the right hand side, beginning
  with the last. Since $\Ba \clu = \ell$, for any $\clv \in \Va$,
  \begin{align*}
    \big( \Ba( \clu - \clwh) \big)(\clv)
    & = \big( \ell - \Ba \clwh \big)(\clv)
     = \ell(\clv) - \ba(\clwh, \clv)
    \\
    & = \ell(\clv - \fpb \clv) - \ba(\clwh, \clv - \fpb \clv)
      + \ell(\fpb \clv) - \ba(\clwh, \fpb \clv)
    \\
    & = \ell(\clv - \fpb \clv) + \big( \ell - \Ba \clwh \big)( \fpb \clv)
    \\
    & \le \big( \osc(\ell)
      + \| \ell - \Ba \clwh \|_{(\Vah)'} \| \fpb \|_{L(\Va)}  \big)
      \| \clv \|_{\Va},
  \end{align*}
  where we have used~\eqref{eq:Fortin-b} in the third step and
  $\fpb \clv \in \Vah$ in the last.  Taking the supremum
  over $0 \ne \clv \in \Va$ and using $(a+b)^2 \le 2 a^2 + 2 b^2$, yields
  \begin{equation}
    \label{eq:rel-bound-B}
    \kappa \| \Ba(\clu - \clwh) \|_{\Vap}^2
    \le 2 \kappa \osc(\ell)^2
    + 2 \kappa \| \fpb \|_{L(\Va)}^2 \| \ell - \Ba \clwh \|_{(\Vah)'}^2.
  \end{equation}

  {\em Step 2.}  Next, to bound the term with the nonlinear operator,
  since $\Ca(\clu) = 0$, by~\eqref{eq:N-pairing},
  \[
    | \langle \Ca(\clu) - \Ca(\clwh), \clu - \clwh \rangle |
    = | \langle \Ca(\clwh), \clu - \clwh \rangle |
    \le \Ma{\clwh} \, \| \clu - \clwh \|_{\Ua}.
  \]
  Applying Young's inequality,
  \begin{equation}
    \label{eq:rel-bound-C}
    | \langle \Ca(\clu) - \Ca(\clwh), \clu - \clwh \rangle |
    \le \frac{1}{2 c_U} \Ma{\clwh}^2
    + \frac{c_U}{2} \| \clu - \clwh \|_{\Ua}^2.
  \end{equation}
  Substituting~\eqref{eq:rel-bound-B}
  and~\eqref{eq:rel-bound-C} into~\eqref{eq:rel-basic}, and absorbing
  the last term of~\eqref{eq:rel-bound-C} into the left hand side,
  \begin{equation}
    \label{eq:7}
    \begin{aligned}
    \frac{c_U}{2} \| \clu - \clwh \|_{\Ua}^2
   & \le \frac{1}{2 c_U} \Ma{\clwh}^2
    + 2 \kappa \| \fpb \|_{L(\Va)}^2 \| \ell - \Ba \clwh \|_{(\Vah)'}^2
    + 2 \kappa \osc(\ell)^2\\
    & \le\max\left(2, \frac{1}{2c_U\|\fpb\|^2_{L(\Va)}}\right)\|\fpb\|^2_{L(\Va)}
    \cl L(\alpha,\clw_h) +2 \kappa \osc(\ell)^2.
    \end{aligned}
  \end{equation}
  Thus, defining 
  \begin{equation}
    \label{eq:8}
    \gamma =
    \max\left( 2, \; \frac{1}{2 c_U \| \fpb \|_{L(\Va)}^2}  \right),
    \qquad
    c = \frac{c_U }{2 \max(\gamma, 2 \kappa)}, 
  \end{equation}
  the reliability estimate~\eqref{eq:reliability-L} holds with the $c$
  in~\eqref{eq:8}.

  {\em Step 3.}  To prove the efficiency
  estimate~\eqref{eq:efficiency-L}, note that since
  $\Vah \subseteq \Va$ and $\Ba \clu = \ell$,
  \[
    \| \ell - \Ba \clwh \|_{(\Vah)'}
    \le \| \ell - \Ba \clwh \|_{\Vap}
    = \| \Ba(\clu - \clwh) \|_{\Vap}
    \le \| \Ba \|_{L(\Ua, \Vap)} \| \clu - \clwh \|_{\Ua},
  \]
  while $\Ca(\clu) = 0$ and~\eqref{eq:N-lip} give
  \[
    \Ma{\clwh}
    \le c_M \| \clu - \clwh \|_{\Ua}.
  \]
  Using these two bounds in the definition~\eqref{eq:C&H-lossh} of
  $\cl L$ proves the efficiency estimate~\eqref{eq:efficiency-L}
  with $C = \kappa \| \Ba \|_{L(\Ua, \Vap)}^2 + c_M^2$.
\end{proof}

\begin{remark}
  \label{rem:quasioptimality}
  As is well known (see, e.g., \cite[Remark~2.5]{CarstDemkoGopal14})
  {\it a posteriori} efficiency and reliability estimates like that of
  Theorem~\ref{thm:gen-nonlinear-dpg} imply {\it a priori}
  quasi-optimal error estimates for any loss minimizer
  $\clu_h \in \arg \min_{\clw \in \Uah}\cl L(\alpha, \clw)$ when
  $\osc(\ell)$ vanishes or is comparable to the best approximation error.
  The argument is
  \begin{align}
    \nonumber
    \| \clu - \clu_h \|_{\Ua}^2
    &\le \frac 1 c
      \cl L(\alpha, \clu_h) \| \fpb \|_{L(\Va)}^2 +
      \frac 1 c\osc(\ell)^2 
    \le \frac 1 c \cl L(\alpha, \clw_h) \| \fpb \|_{L(\Va)}^2 +
      \frac 1 c \osc(\ell)^2 \\
    \label{eq:quasioptimality}
    &\le \left(\frac C c \| \fpb \|_{L(\Va)}^2 \right)
      \| \clu - \clw_h \|_{\Ua}^2 
      + \frac 1 c \osc(\ell)^2,
  \end{align}
  for every $\clw_h \in \Uah$, using
  reliability~\eqref{eq:reliability-L}, minimality, and
  efficiency~\eqref{eq:efficiency-L} in succession.
\end{remark}

\subsection{The least-squares instance}
\label{ssec:mfosls-abstract}

Consider the choice $\Vah = \Va$ and $\fpb = \mathrm{id}$, which, as
already noted, satisfies Assumption~\ref{asm:fortin-b} trivially with
$\| \fpb \|_{L(\Va)} = 1$, and which makes $\osc(\ell) = 0$ for any
$\ell \in \Vap$.  For this choice,
Theorem~\ref{thm:gen-nonlinear-dpg} yields
\begin{equation}
  \label{eq:ls-vc}
  c\, \| \clu - \clw \|_{\Ua}^2
  \;\le\; \cl L(\alpha, \clw)
  \;\le\; C\, \| \clu - \clw \|_{\Ua}^2,
  \qquad \clw \in \Uah,
\end{equation}
with the same constants $c, C$ there, and the trial subspace $\Uah$
may be taken to be all of $\Ua$. Indeed,
no mesh ever appears  in the analysis of this instance, so $\clw$ in~\eqref{eq:ls-vc} 
can even  be taken from a meshless NN  hypothesis class within~$\Ua$.
The first term of the loss is
then $\kappa \| \ell - \Ba \clwh \|_{\Vap}^2$, which is computable
whenever the Riesz map $\Rza$ is explicitly invertible: e.g. as mentioned 
earlier, when $\Va$ is an $L_2$-type space, identified with its own dual, 
the term is a plain integral of the squared strong-form residual.  
Complemented by a computable choice of $\Ma\cdot$ (see~\eqref{eq:mfosls-N} 
for an example),  every term of the loss then becomes directly computable 
(without any element-local inversions as in the next option of 
\S\ref{ssec:mdpg-abstract}).  Since the examples we have in mind all have 
first order differential operators in $\Ba$, we think of this loss function 
as a nonlinear first-order system least-squares (FOSLS) functional built  
into the above monotone framework. Accordingly, we call this instance the 
{\bf{MFOSLS loss function}}.  A concrete realization for a nonlinear
diffusion problem is given in~\S\ref{sec:fosls-loss-function}.

\subsection{The DPG instance}
\label{ssec:mdpg-abstract}

In the DPG methodology, one seeks a variational formulation with 
{\em broken test spaces} (a product of test spaces associated with mesh 
elements), localizing the Riesz map of $\Va$. It then becomes computable
by reducing each component to a finite-dimensional space (depending on
$\Uah$), rendering $\Vah$ finite-dimensional.
More precisely,
we pick a conforming finite element mesh $\oh$ from a hierarchy of
uniformly shape-regular meshes (of mesh size $h$) and set the
``broken'' infinite-dimensional test space $\Va$, and a subspace
$\Vah$ for practical computations, by the Cartesian products
\begin{equation}
  \label{eq:V-broken'}
  \Va = \bigtimes_{K\in\oh}\VK, \qquad  
  \Vah = \bigtimes_{K\in\oh}\VhK,
\end{equation}
where $\VhK \subseteq \VK$ are finite dimensional.  The (infinite
dimensional) trial space $\Ua$ may also be dependent on the mesh
$\oh$, but our ensuing analysis will produce, under
Assumption~\ref{asm:fortin-b}, {\em mesh-independent} estimates
throughout.
The first term of the loss,
$\| \ell - \Ba \clz_h \|_{(\Vah)'}^2$, is locally computable through
the discrete Riesz representation of the residual component in the
broken space $\Vah$. Namely, given an arbitrary $\clw_h \in \Uah$, let
$\Ea{b}{\clwh}$ be the unique function in $\Vah$ satisfying
\begin{equation}
  \label{eq:residual-rep}
  (\Ea{b}{\clwh}, \clv_h)_\Va
  = \ell(\clv_h) - \ba(\clwh, \clv_h), \qquad \clv_h \in \Vah.
\end{equation}
Since $\Vah$ is broken---see~\eqref{eq:V-broken'}---the function
$\Ea{b}{\clwh}$ can be computed element by element by inverting small
Gram matrices of the $\VhK$-inner product on each
$K \in \oh$.  Moreover, since
$\| \ell - \Ba \clz_h \|_{(\Vah)'}$ equals the supremum
of the right hand side of~\eqref{eq:residual-rep}, divided by
$\| \clvh \|_{\Va}$, over all $0 \ne \clvh \in \Vah$, we have
\begin{equation}
  \label{eq:L-and-E-b}
  \cl L(\alpha, \clz_h)
  = \kappa \| \Ea{b}{\clzh} \|_{\Va}^2
  + \Ma{\clz_h}^2.
\end{equation}
In practice, we compute $\cl L(\alpha, \clz_h)$ by first computing
$\Ea{b}{\clzh}$ using~\eqref{eq:residual-rep} and then
using~\eqref{eq:L-and-E-b}, which also permits easy differentiation of
the loss function with respect to the parameter.  For the second term
we take the canonical choice $\Ma{\clz} = \| \Ca(\clz) \|_{\Uap}$,
which obviously  meets Assumption~\ref{asm:N}.  This term is
easily computable only in situations where
$\| \Ca(\clz_h) \|_{\Uap}$ is computable, and we will see in examples
that we must design the $\Ca$-term to enable this.  We call the
resulting instance the {\bf{MDPG loss function}}, or the monotone DPG
loss function.  A concrete realization is given
in~\S\ref{sec:mdpg-loss-function-diffus}.

\section{Parametric diffusion with nonlinear fluxes}
\label{sec:nonlinear-diffusion}

A model nonlinear diffusion problem~\cite{carstensen2017, Zeidl90} on
a unit-sized  open bounded set $\om \subset \bb R^d$ arises from minimizers of
nonlinear ``energy functionals'' of the form
\begin{equation}
  \label{eq:min-with-varphi}
  u = \argmin_{w \in \Ho^1(\om)} \int_\om \varphi_\alpha(|\nabla w|) - f w
\end{equation}
where $f \in L_2(\om)$, and $\varphi_\alpha(t)$, a positive nonlinear function
for $t \in [0, \infty)$ and $\alpha$ in some parameter domain $\cl Q$ to be specified shortly, are given.  Here $| \cdot |$ denotes the Euclidean
length. We proceed assuming, as in \cite{carstensen2017}, that there
is a function $\phi_\alpha \in C^2(0, \infty)$ such that 
\begin{equation*}
  \varphi_\alpha(t) = \int_{0}^{t} s \,\phi_\alpha(s) \; ds.
\end{equation*}
Then, since $\varphi_\alpha'(t) = t \phi_\alpha(t)$, the Euler-Lagrange equations
of~\eqref{eq:min-with-varphi} show that $u \in \Ho^1(\om)$ satisfies
\begin{equation}
  \label{eq:3}
  \int_\om  \phi_\alpha(|\nabla u |)\, \nabla u \cdot \nabla v
  = \int_\om f v, \qquad  v \in \Ho^1(\om).
\end{equation}
Taking cue from this variational form, assume that the nonlinear flux
function $\sigma_\alpha : \bb R^d \to \bb R^d$ is invertible, i.e.,
there is a $\rho_\alpha: \bb R^d \to \bb R^d$ such that
\begin{equation}
  \label{eq:5}
  \sigma_\alpha(q) = -\phi_\alpha( |q|) q,
  \qquad
  \rho_\alpha(\sigma_\alpha(q)) = q, 
  \qquad q: \om \to \bb R^d,
\end{equation}
(An explicit
formula for $\rho_\alpha$ is needed only for the dual
formulation~\eqref{eq:4-mixed} below and for the alternative MDPG formulation
of Remark~\ref{rem:mdpg-rho-formulation}; the MDPG and MFOSLS loss functions
use $\sigma_\alpha$ only.)

Equation~\eqref{eq:3}  corresponds to the classical boundary value
problem to find $u$ satisfying
\begin{equation}
  \label{eq:4}
  \begin{aligned}
    \dive (\sigma_\alpha( \nabla u)) & = f  && \text{ in } \om,
    \\
    u & = 0 && \text{ on } \partial \om.
  \end{aligned}
\end{equation}
In first-order form, introducing the flux of the solution $u$ as a new
variable $q= \sigma_\alpha(\nabla u)$ and inverting,
$\rho_\alpha(q) = \nabla u$, equation~\eqref{eq:4} becomes
\begin{subequations}
  \label{eq:4-mixed}
  \begin{align}
    \label{eq:4-mixed-a}
    \rho_\alpha(q) - \nabla u &= 0, \quad\text{in }\om, \\
    \label{eq:4-mixed-b}
    \operatorname{div} q &= f, \quad\text{in }\om, \\
    u &= 0, \quad \text{on } \partial \om
  \end{align}
\end{subequations}
When computing with the primal formulation~\eqref{eq:4}, we need
$\sigma$ explicitly. In contrast, when computing with the dual
formulation~\eqref{eq:4-mixed}, we need $\rho$ explicitly.

We consider DPG discretizations using a shape-regular simplicial finite
element mesh $\oh$ of $\om \subset \bb R^d$.  Recall the standard
spaces 
$\Ho^1 = 
\{ w \in L_2(\om): \partial_i w \in L_2(\om), w|_{\partial\Omega}=0\}$ 
and
$\Hdiv = \{ r \in L_2(\om)^d: \dive r \in L_2(\om)\}$.
Denote by $n$
the outward unit normals on element boundaries, and let $\tr(w)$ and
$\trn(r)$ denote element-wise trace operators characterized as
$\tr(w)|_{\partial K} = w|_{\partial K}$ and
$\trn(r)|_{\partial K} = (r \cdot n )|_{\partial K}$. By standard
trace theorems, $\tr(\cdot)$ can be applied to $w \in \Ho^1$ and
$\trn(\cdot)$ to $q \in \Hdiv$.  Let $\Hoh = \tr \Ho^1$ and
$\Hmh = \trn \Hdiv$, endowed with the quotient norms
\begin{equation}
  \label{eq:quotient-norms}
  \|\hat u\|_{{\Hoh}}
  := \inf_{u \in \text{tr}^{-1}\{\hat u\}} \|u\|_{\Ho},
  \qquad
  \|\hat r_n\|_{H^{-1/2}(\partial \Omega_h)}
  := \inf_{ q \in \text{tr}_n^{-1}\{\hat r_n\}} \| q\|_{H(\div)},
\end{equation}
respectively. In addition to these spaces, we also need the  product spaces
\begin{equation}
  \label{eq:13}
  \Hdivoh = \bigtimes_{K \in \oh} H(\div, K), \qquad
  H^1(\oh) = \bigtimes_{K \in \oh} H^1(K),  
\end{equation}
representing piecewise $H(\div)$ and $H^1$ spaces. Functions in those 
spaces do not have any interelement continuity constraints. When computing 
differential operators on them, we do so piecewise, and use an  abbreviated 
notation for sums over elements, 
\[
  (\div r, w)_{\oh} = \sum_{K \in \oh} \int_K w \div r,\qquad
  (r, \nabla w)_{\oh} = \sum_{K \in \oh} \int_K r \cdot \nabla w, 
  \qquad 
\]
for all $r \in \Hdivoh, \; w \in H^1(\oh),$ where the derivatives are
computed element by element. Accordingly, their norms are given by
\begin{equation}
  \label{eq:Hdivoh-h1oh-norms}
  \| r \|_{\Hdivoh}^2 = (r, r)_\oh + (\dive r, \dive r)_\oh, \quad
  \| w \|_{H^1(\oh)}^2 = (w, w)_\oh + (\nabla w, \nabla w)_\oh.
\end{equation}
For
$\hat r_n \in H^{-1/2}(\partial \Omega_h), r \in \Hdivoh, w \in H^1(\oh),
\hat w \in \Hoh$,  the  abbreviated notation
\[
  \ip{ \hat r_n, w }_{\partial \oh} = \sum_{K \in \oh} \ip{ \hat r_n,
    w}_{H^{1/2}(\partial K)}, \quad
  \ip{ r \cdot n, \hat w}_{\partial\oh} =
  \sum_{K \in \oh} \ip{ r\cdot n, \hat w}_{H^{1/2}(\partial K)},
\]
for such frequently occurring sums over element boundaries will be useful.

\subsection{The MDPG loss function}
\label{sec:mdpg-loss-function-diffus}

We now show how to  fit also~\eqref{eq:4-mixed} into the framework for
the MDPG loss function.  Writing the gradient as a new unknown
$p := \nabla u$, problem~\eqref{eq:4} becomes the following system of
three equations
\begin{equation}
\label{eq:9}
  p = \nabla u, \quad
  \dive q = f, \quad
  q = \sigma_\alpha(p)
\end{equation}
with $u$ subject to the homogeneous Dirichlet boundary condition
built into $\Ho^1(\om)$.
In other words, defining
$B : 
\Hdiv \times \Ho^1(\om) \times L_2(\om) \to L_2(\om) \times L_2(\om)^d$ 
and
$C_\alpha : \Hdiv  \times \Ho^1(\om) \times L_2(\om) \to L_2(\om)^d$
by
\[
  B(q, u, p) 
  =
  (\dive q, p - \nabla u), 
  \qquad
  C_\alpha(q, u, p)
  =
  q  - \sigma_\alpha(p)
\]
the problem~\eqref{eq:4} is equivalent to finding
$(q, u, p) \in \Hdiv  \times \Ho^1(\om) \times L_2(\om)$ satisfying
\begin{equation}
  \label{eq:11}
    B(q, u, p) = (f, 0), \qquad C_\alpha(q, u, p) = 0.
\end{equation}

To impose the first equation of~\eqref{eq:11} in an ultraweak manner,
we multiply it by a broken test function tuple
$(r, w ) \in  \Hdivoh \times H^1(\oh) $ and integrate by parts
element by element, resulting in the following formulation of~\cite{CH}:
\begin{subequations}
  \label{eq:B-C-example}
\begin{equation}
  \label{eq:16}
  \begin{aligned}
    b_\alpha(\clu, \clv) & = \ell(\clv),
    && \text{ for all } \clv \in \Va,   
    \\
    c_\alpha(\clu; \tilde\clu) & = 0,
    && \text{ for all } \tilde\clu \in \Ua,   
  \end{aligned}
\end{equation}
where the spaces and forms are as set below, $\clv = (r, w)$ is in
$ \Va$, $\clu = (q, u, p, \hat u, \hat q_n)$ and
$\tilde\clu = (\tilde q, \tilde u, \tilde p, \tilde{\hat u},
\tilde{\hat q}_n)$ are in $ \Ua$,
\begin{align}
  \label{eq:Ua-MDPG}
  \Ua          
  & :=  L_2(\om)^d \times L_2(\om) \times L_2(\om)^d \times \Hoh \times \Hmh, \quad
  \\
  \Va
  & :=    \Hdivoh \times H^1(\oh),
    \quad \qquad \ell( \clv) :=  (f, w)_{L_2},
  \\
  \label{eq:b-alpha}
  b_\alpha(\clu, \clv)
  & :=
     \ip{ \hat q_n, w}_{\partial \oh} - \ip{ r \cdot n, \hat u}_{\partial \oh}
    + (p, r)_\oh + (u, \div r)_\oh -(q, \nabla w)_\oh
  \\ \label{eq:c-alpha}
  c_\alpha(\clu; \tilde \clu)
       & := \kappa_1^{-2} (q - \sigma_\alpha(p), \qt)_{L_2}
         + \kappa_1^{-2} \kappa_2 (q - \sigma_\alpha(p), \pt)_{L_2}
\end{align}
for some constants $\kappa_1, \kappa_2 > 0$.  Again
following~\cite{CH}, we set these weighted norms for
$\clu = (q, u, p, \hat u, \hat q_n)$ and $\clv = (r, w)$:
\begin{equation}
  \label{eq:mdpg-norms}
  \begin{aligned}
    \|\clu  \|_{\Ua}^2
    & := \| u \|_{L_2}^2 + \| p \|_{L_2}^2      
      + \| \hat u \|^2_{H^{1/2}(\partial \oh)}
      + \kappa_1^{-2}
      \left( \| q \|_{L_2}^2 + 
      \| \hat q_n \|^2_{H^{-1/2}(\partial\oh)} \right),
    \\
    \| \clv \|_{\Va}^2
    & :=  \kappa^2_1\,  \| w\|^2 _{H^1(\oh)}
      + \| r \|^2_{\Hdivoh}.
  \end{aligned}
\end{equation}
The operators $\Ba$ and $\Ca$ in~\eqref{eq:A-split-BandC} are defined
by the $b_\alpha$ and $c_\alpha$ above. Note that $\Ba$ is purely
linear and, in this case, has no $\alpha$ dependence. Now, for any integer $k\ge 0$, setting $U_{k+1} = {P}_{k+1} \cap {\Ho}^1$, $\hat U_{k+1} = \tr U_{k+1}$,  $\RT_k = \Hdiv \cap (P_k^d + x P_k)$,   $\hat\RT_k = \trn (\RT_k)$, and 
\begin{equation}
  \label{eq:mdpg-discrete}
  \Uah = P_k^d \times P_k \times P_k^d \times  \hat U_{k+1} \times \hat\RT_k,
  \qquad
  \Vah = P_{k+2}^d \times P_{k+d}
\end{equation}
\end{subequations}
as discrete subspaces of $\Ua$ and $\Va$, respectively, all
ingredients to compute the MDPG loss function of
Definition~\ref{defn:mdpg} are in place.

\begin{remark}[Alternative MDPG formulation]
  \label{rem:mdpg-rho-formulation}
  The MDPG formulation above uses the auxiliary variable $p$ and enforces
  the closure relation $q = \sigma_\alpha(p)$. An alternative formulation
  replaces $p$ with an auxiliary variable $z \in L_2(\oh)^d$ and enforces
  the inverse closure relation $z = \rho_\alpha(q)$, where
  $\rho_\alpha = \sigma_\alpha^{-1}$. The trial space becomes
  $\Ua = L_2(\oh)^d \times L_2(\oh) \times L_2(\oh)^d
  \times \Hoh  \times H^{-1/2}(\partial\oh)$ with variables
  $(q, u, z, \hat u, \hat q_n)$, and the closure operator is
  \[
    c_\alpha(\clu; \tilde\clu)
    := \kappa_1^{-2}(z - \rho_\alpha(q), \tilde z)_{L_2}
    + \kappa_1^{-2}\kappa_2(z - \rho_\alpha(q), \tilde q)_{L_2}.
  \]
  The bilinear form $b_\alpha$ remains unchanged. Both formulations
  share the same discrete subspaces~\eqref{eq:mdpg-discrete} and yield
  identical variational correctness constants (see
  Remark~\ref{rem:mdpg-rho-verification}).
\end{remark}

To understand the computability of the $\Ca$ term in the MDPG loss
function, observe that \eqref{eq:c-alpha} implies 
\begin{equation}
  \label{eq:C-mdpg'}
  \langle \Ca(\clu), \tilde \clu \rangle
  = \kappa_1^{-2}
  \big( q - \sigma_\alpha(p), \; \qt + \kappa_2 \pt \big)_{L_2}.
\end{equation}
Hence maximizing~\eqref{eq:C-mdpg'} separately in $\pt$ and
$\qt$ using  the weighted norm~\eqref{eq:mdpg-norms} gives
the closed-form expression
\begin{equation}
  \label{eq:C-dualnorm-closedform}
  \| \Ca(\clu) \|_{\Uap}
  = \kappa_1^{-2} \big( \kappa_1^2+  \kappa_2^2\big)^{1/2}
  \, \| q - \sigma_\alpha(p) \|_{L_2},
\end{equation}
which, being an $L_2$-norm, is computable.  Thus the canonical choice
$\Ma{\clz} = \| \Ca(\clz) \|_{\Uap}$ made by the MDPG loss function
in~\S\ref{ssec:mdpg-abstract} is computable in this setting.

\subsection{The MFOSLS loss function}
\label{sec:fosls-loss-function}

The first-order reformulation~\eqref{eq:11} can alternately be kept in
strong form, leading to a realization of the MFOSLS instance mentioned 
earlier (in~\S\ref{ssec:mfosls-abstract}).  Set
\begin{subequations}
  \label{eq:mfosls-setting}
\begin{equation}
  \label{eq:mfosls-spaces}
  \Ua := \Hdiv \times \Ho^1(\om) \times L_2(\om)^d,
  \qquad
  \Va := L_2(\om) \times L_2(\om)^d,
\end{equation}
and for $\clz = (q, u, p) \in \Ua$ and $\clv = (w, r) \in \Va$, let
\begin{equation}
  \label{eq:mfosls-b}
  \ba(\clz, \clv) := (\dive q, w)_{L_2} + (p - \nabla u, r)_{L_2},
  \qquad
  \ell(\clv) := (f, w)_{L_2},
\end{equation}
i.e., $\Ba$ now equals the operator $B$ of~\eqref{eq:11}.
The nonlinear closure $\Ca$ is generated by the
same formula~\eqref{eq:c-alpha} used for MDPG (whose right hand side
only involves the $q$- and $p$-components of the trial tuples).
Mirroring the weighted norms~\eqref{eq:mdpg-norms}, we set
\begin{equation}
  \label{eq:mfosls-norms}
  \begin{aligned}
    \| \clz \|_{\Ua}^2
    & := \| u \|_{L_2(\om)}^2 +  \| \nabla u \|^2_{L_2}
      + \| p \|_{L_2}^2
      + \kappa_1^{-2}
      \left( \| q \|_{L_2}^2 + \| \dive q \|^2_{L_2} \right),
    \\
    \| \clv \|_{\Va}^2
    & := \kappa_1^2 \, \| w \|_{L_2}^2 + \| r \|^2_{L_2},
  \end{aligned}
\end{equation}
\end{subequations}
with the same constants $\kappa_1, \kappa_2 > 0$
entering~\eqref{eq:c-alpha}.  No mesh, broken space, or interface
variable is needed.  Since $\Va$ is an $L_2$ product identified with
its own dual,
\begin{equation}
  \label{eq:mfosls-first-term}
  \| \ell - \Ba \clz \|_{\Vap}^2
  = \kappa_1^{-2} \| f - \dive q \|_{L_2}^2
  + \| p - \nabla u \|_{L_2}^2
\end{equation}
is exactly computable.  For the closure term, we choose
\begin{equation}
  \label{eq:mfosls-N}
  \Ma{\clz}
  := \kappa_1^{-2} \big( \kappa_1^2 + \kappa_2^2 \big)^{1/2}
  \, \| q - \sigma_\alpha(p) \|_{L_2},
\end{equation}
i.e., the same closed-form expression that the dual norm produced in
the MDPG setting---see~\eqref{eq:C-dualnorm-closedform}.  Note that
$\| \Ca(\clz) \|_{\Uap}$ need not equal~\eqref{eq:mfosls-N} here: the
trial norm in~\eqref{eq:mfosls-norms} is stronger than the
corresponding weighted $L_2$ product norm, so \eqref{eq:mfosls-N} is
in general only an upper bound for the dual norm, as seen in the
verification of Assumption~\ref{asm:N} for the
choice~\eqref{eq:mfosls-N} in Proposition~\ref{prop:mfosls-verif}
below.  With $\Vah = \Va$ and $\fpb = \mathrm{id}$, the MFOSLS loss
function, generated per Definition~\ref{defn:mdpg}, reads
\begin{equation}
  \label{eq:mfosls-loss}
  \cl L(\alpha, \clz)
  = \kappa \left( \kappa_1^{-2} \| f - \dive q \|_{L_2}^2
    + \| p - \nabla u \|_{L_2}^2 \right)
  + \kappa_1^{-4} \big( \kappa_1^2 + \kappa_2^2 \big)
  \, \| q - \sigma_\alpha(p) \|_{L_2}^2,
\end{equation}
for $\clz = (q, u, p) \in \Ua$.  Every term is a plain $L_2$ integral.
The difference, compared to MDPG, is conformity: trial functions must
satisfy $q \in \Hdiv$ and $u \in \Ho^1(\om)$. Discrete trial spaces
such as
\begin{equation}
  \label{eq:mfosls-discrete}
  \Uah = \RT_k \times U_{k+1} \times P_k^d,
\end{equation}
where $\RT_k$ and $U_{k+1}$ are as defined before~\eqref{eq:mdpg-discrete},
provide conforming choices,
but their choice is immaterial for the reliability and efficiency,
since MFOSLS variational correctness holds throughout~$\Ua$.

\section{Verification of the assumptions for a class of examples}
\label{sec:verification}

\subsection{A model class of nonlinear diffusive fluxes}
\label{ssec:model}

In this section we verify the assumptions of
Theorem~\ref{thm:gen-nonlinear-dpg} for the MDPG and MFOSLS
formulations of \S\ref{sec:mdpg-loss-function-diffus}
and~\S\ref{sec:fosls-loss-function} with a specific
nonlinearity obtained by setting 
\begin{subequations}
  \label{eq:phi-asm}
\begin{equation}
  \label{eq:10}
  \phi_\alpha := \alpha^{-1}\phi
\end{equation}
 where  $\phi$ is some positive function satisfying for some constants $0 < \beta_1 < \beta_2 < \infty$
\begin{equation}
  \label{eq:phi-props}
  0 < \beta_1 \leq \phi(t) \leq \beta_2,
  \quad \beta_1 \leq \phi(t) + t \phi'(t) \leq \beta_2, \qquad
  \text{ for all } t>0, \qquad \lim_{t \to 0^+} t\phi'(t) = 0.
\end{equation}
In~\eqref{eq:10}, $\alpha$ is from a collection $\cl Q$ of
bounded positive functions on $\om$
(representing some material coefficients) and we assume that there are some fixed
constants $\alpha_1, \alpha_2$ such that
\begin{equation}
  \label{eq:15}
  \alpha_1 \le \alpha(x) \le \alpha_2, 
\end{equation}
\end{subequations}
for a.e.~$x \in \om$ for all $\alpha \in \cl Q$. This setting of~\eqref{eq:phi-asm} then fixes
the energy by \eqref{eq:min-with-varphi}, the nonlinear flux
$\sigma_\alpha(q) = -\phi_\alpha( |q|) q$ as in \eqref{eq:5}, and the parametric boundary value problem by~\eqref{eq:4}.

\begin{example}[Typical profile]
  \label{ex:canonical}
  Set $\phi(t) = 2 + (1 + t)^{-2}$ for $t\ge 0$.  Then, since
  $\phi'(t) = -2(1+t)^{-3}$, we have $2 < \phi(t) \le 3$ and
  \[
    \phi(t) + t \phi'(t)
    = 2 + \frac{1+t-2t}{(1+t)^3} = 2 + \frac{1-t}{(1+t)^3}.
  \]
  The function $t \mapsto (1-t)(1+t)^{-3}$ has derivative
  $(2t-4)(1+t)^{-4}$, hence attains its minimum $-1/27$ at $t=2$ and
  its maximum $1$ at $t = 0$.  Therefore
  $2 - 1/27 \le \phi(t) + t\phi'(t) \le 3$; as $t\phi'(t) \to 0$
  when $t \to 0^+$, \eqref{eq:phi-props}
  certainly holds with $\beta_1 = 1$, $\beta_2 = 3$.  The form of
  $\phi$ is found in \cite[Example~3.2(a)]{carstensen2017} as well as
  in~\cite{CH}.  The numerical experiments of~\cite{CH} use the
  variant $\phi(t) = 2 - (1+t)^{-2}$ for which a similar analysis
  holds.
\end{example}

\begin{example}[Invertible profile]
  \label{ex:invertible}
  Set $\phi(t) = 2 + (1+t)^{-1}$ for $t \ge 0$. Then $\phi'(t) = -(1+t)^{-2}$,
  so $2 < \phi(t) \le 3$ and
  \[
    \phi(t) + t\phi'(t)
    = 2 + \frac{1}{1+t} - \frac{t}{(1+t)^2}
    = 2 + \frac{1+t-t}{(1+t)^2}
    = 2 + (1+t)^{-2}.
  \]
  Hence $2 < \phi(t) + t\phi'(t) \le 3$, and since $t\phi'(t) \to 0$
  when $t \to 0^+$, \eqref{eq:phi-props} holds with
  $\beta_1 = 2$ and $\beta_2 = 3$.

  This profile has the additional property that $\sigma_\alpha$ is explicitly
  invertible. To see this, observe that the monotone function
  $m(t) := t\phi(t) = t\left(2 + \frac{1}{1+t}\right) = 2t + \frac{t}{1+t}$
  can be rewritten as
  \[
    m(t) = 2t + 1 - \frac{1}{1+t},
  \]
  from which $m(t) = s$ gives $(1+t)^{-1} = 2t + 1 - s$, or equivalently
  \[
    2t^2 + (3-s)t - s = 0.
  \]
  Solving this quadratic equation yields
  \[
    m^{-1}(s)
    = \frac{s-3+\sqrt{(3-s)^2+8s}}{4}
    = \frac{s-3+\sqrt{s^2+2s+9}}{4}
  \]
  for $s \ge 0$ (taking the positive root since $t \ge 0$). One can verify
  that $m(m^{-1}(s)) = s$ by direct substitution. To summarize, 
  if $p = \sigma_\alpha(q)$ then $|p| = \alpha^{-1}m(|q|)$,
  so $|q| = m^{-1}(\alpha|p|)$ and the direction reversal gives  the
  explicit inverse
  \begin{equation}
    \label{eq:rho-explicit}
    \rho_\alpha(p) = \sigma_\alpha^{-1}(p)
    = -\frac{m^{-1}(\alpha|p|)}{|p|}\, p
    = -\frac{\alpha|p| - 3 + \sqrt{\alpha^2|p|^2 + 2\alpha|p| + 9}}{4|p|}\, p.
  \end{equation}
  The explicit inverse enables the alternative MDPG formulation of
  Remark~\ref{rem:mdpg-rho-formulation} for this profile.
\end{example}

\begin{remark}
  For the typical profile of Example~\ref{ex:canonical} (with
  $\phi(t) = 2 + (1+t)^{-2}$), it is not as easy as
  Example~\ref{ex:invertible} to obtain the inverse of the nonlinear flux.
  It would require  solving a cubic equation: in contrast, the invertible
  profile of Example~\ref{ex:invertible} reduces to a quadratic, enabling
  explicit inversion at the cost of slightly weaker regularity
  ($\phi'(t) \sim (1+t)^{-2}$ vs.\ $(1+t)^{-3}$ for the typical case).
  Both profiles satisfy~\eqref{eq:phi-props} with comparable constants.
\end{remark}

We
start with the following lemma, proved using arguments
from~\cite{carstensen2017}. In particular,~\eqref{eq:sigma-mono-lip}
below shows the monotone form of the nonlinearity in $-\sigma_\alpha$
which  will enable us to establish the strong-monotonicity
estimate~\eqref{eq:S4-mono}.

\begin{lemma}
  \label{lem:phi}
  In the setting of~\eqref{eq:phi-asm}, the nonlinear flux function
  $\sigma_\alpha$ of~\eqref{eq:5} which maps
  $q \mapsto -\phi_\alpha(|q|)q$ is $C^1$ on $\bb R^d$, its Jacobian matrix 
  at any $0 \ne q \in \bb R^d$ takes the form 
  \begin{equation}
    \label{eq:Dsigma}
    d\sigma_\alpha(q)
    = -\alpha^{-1} \left( \phi(|q|)\, I
      + \phi'(|q|)\, |q| \; \frac{q}{|q|} \otimes \frac{q}{|q|}\right),
  \end{equation}
  which is extendable continuously to all $q \in \bb R^d$  including $q = 0$ and 
  has its spectrum
  contained in the real interval  $[-\alpha^{-1}\beta_2, -\alpha^{-1}\beta_1]$ for every
  $q \in \bb R^d$.
  Moreover, for all $q, r \in \bb R^d$,
  \begin{equation}
    \label{eq:sigma-mono-lip}
    \begin{gathered}
      \alpha^{-1}\beta_1 |q - r|^2
      \;\le\; -\big(\sigma_\alpha(q) - \sigma_\alpha(r)\big)\cdot(q-r)
      \;\le\;\alpha^{-1}\beta_2 |q-r|^2,
      \\
      \alpha^{-1}\beta_1 |q-r|
      \;\le\; |\sigma_\alpha(q) - \sigma_\alpha(r)|
      \;\le\; \alpha^{-1}\beta_2 |q - r|.
    \end{gathered}
  \end{equation}
  Consequently, the inverse $\rho_\alpha = \sigma_\alpha^{-1}$ satisfies
  \begin{equation}
    \label{eq:rho-bilip}
    \alpha\beta_2^{-1}|p - s|
    \;\le\; |\rho_\alpha(p) - \rho_\alpha(s)|
    \;\le\; \alpha\beta_1^{-1}|p - s|
    \qquad \text{for all } p, s \in \bb R^d.
  \end{equation}
\end{lemma}
\begin{proof}  
  The derivative in~\eqref{eq:Dsigma} is computed easily by chain and
  product rules.  Since $\phi'(|q|)|q|$ vanishes as $q \to 0$
  by~\eqref{eq:phi-props},
  $d\sigma_\alpha$ is continuous on all of $\bb R^d$.
  The symmetric matrix in~\eqref{eq:Dsigma}
  equals $-\alpha^{-1}$ times a matrix with eigenvalue
  $\phi(|q|) + |q| \phi'(|q|)$ in the direction of $q$ and eigenvalue
  $\phi(|q|)$ on the orthogonal complement, both of which lie in
  $[\beta_1,\beta_2]$ by~\eqref{eq:phi-props},  so 
  $d\sigma_\alpha(q)$ has spectrum contained in
  $[-\alpha^{-1}\beta_2, -\alpha^{-1}\beta_1]$ as claimed.  The
  bound~\eqref{eq:sigma-mono-lip} follows from this and the
  fundamental theorem of calculus along the segment from $r$ to $q$,
  since
  \[
    \big( \sigma_\alpha(q) - \sigma_\alpha(r) \big) \cdot (q - r)
    = \int_0^1 (q-r) \cdot d\sigma_\alpha(r + s(q-r))\, (q-r)\, ds
  \]
  lies in $[-\alpha^{-1}\beta_2 |q-r|^2, -\alpha^{-1}\beta_1|q-r|^2]$.
  The upper Lipschitz bound $|\sigma_\alpha(q)-\sigma_\alpha(r)| \le
  \alpha^{-1}\beta_2|q-r|$ follows since $\|d\sigma_\alpha(\cdot)\|
  \le \alpha^{-1}\beta_2$ everywhere
  (cf.\ \cite[Remark~3.1 and~(4.5)]{carstensen2017}).
  The lower Lipschitz bound follows by combining the strong monotonicity
  inequality with the Cauchy--Schwarz inequality
  $-(\sigma_\alpha(q)-\sigma_\alpha(r))\cdot(q-r) \le
  |\sigma_\alpha(q)-\sigma_\alpha(r)|\,|q-r|$ and dividing by $|q-r|$.
  The bounds~\eqref{eq:rho-bilip} on $\rho_\alpha$ follow
  from~\eqref{eq:sigma-mono-lip}
  upon setting $q=\rho_\alpha(p)$ and $r=\rho_\alpha(s)$ and 
  then rearranging.
\end{proof}

Applying Lemma~\ref{lem:phi} with $\alpha \equiv \alpha(x)$ and
using~\eqref{eq:15}, we find that, for a.e.~$x \in \om$, the flux
$\sigma_\alpha(x, q) = -\alpha^{-1}(x) \phi(| q|)\, q$ satisfies
\begin{gather}
  \label{eq:globalized-sigma}
  \frac{\beta_1}{\alpha_2} |q-r|^2
  \le -\big(\sigma_\alpha(x,q) - \sigma_\alpha(x,r)\big)\cdot(q-r)
  \le \frac{\beta_2}{\alpha_1} |q-r|^2,
  \\
  \frac{\beta_1}{\alpha_2}|q-r|
  \le |\sigma_\alpha(x,q) - \sigma_\alpha(x,r)| \le \frac{\beta_2}{\alpha_1} |q-r|,
  \\
  \label{eq:globalized-rho}
  \frac{\alpha_1}{\beta_2}|p-s|
  \le |\rho_\alpha(x,p) - \rho_\alpha(x,s)| \le \frac{\alpha_2}{\beta_1}|p-s|.
\end{gather}
With this in view,  we set 
\begin{subequations}
  \label{eq:kappas-mu}
\begin{equation}
  \label{eq:kappa-alpha}
  \kappa_1 = \left( \frac{\beta_2}{\alpha_1} \right)^2,
  \qquad
  \kappa_2 = \mu \, \frac{\beta_1}{\alpha_2},
\end{equation}
for some fixed $\mu$ satisfying
\begin{equation}
  \label{eq:20}
  \mu > \left( \frac{\beta_2 \alpha_2}{\beta_1 \alpha_1} \right)^2.
\end{equation}
\end{subequations}
The weighted norms in~\eqref{eq:mdpg-norms} and the definition of $\Ca$ in
\eqref{eq:c-alpha} for the present
class of examples use  these $\kappa_1,\kappa_2$.
By~\eqref{eq:phi-props} $\beta_1 \le \beta_2$ and by~\eqref{eq:15} $\alpha_1 \le \alpha_2$. Hence the quantity on the right hand side of~\eqref{eq:20} is greater than or equal to~$1$, so that $\mu >1$.

\subsection{Variational correctness of the MDPG loss function}
\label{ssec:mdpg-verif}

To prove the efficacy of the MDPG loss function generated per
Definition~\ref{defn:mdpg},
we now verify Assumption~\ref{asm:MDPG} for the operators $\Ba$ and
$\Ca$ generated by $b_\alpha$ and $c_\alpha$
of~\eqref{eq:B-C-example}. This verification proceeds by arguments
that are in~\cite{CH}, but we include a brief sketch for readability.
We will use the existence of a mesh-independent constant $c_P>0$ such
that for all $w \in \Ho^1(\om)$ and $v \in H^1(\oh)$,
\begin{subequations}
  \label{eq:Friedrichs2}
\begin{gather}
  \label{eq:Friedrichs-std}
  c_P \| w \|_{L_2}^2
  \le  \|\nabla w \|_{L_2}^2,
  \\
  \label{eq:Friedrichs-generalized}
  c_P \| v \|_{L_2}^2
  \le
  \sup_{q \in \Hdiv} \left(
    \frac{ \ip{ q\cdot n, v}_{\partial \oh}}
    { \| q \|_{\Hdiv}} \right)^2 
   + \sum_{K\in \oh} \|\nabla v \|_{L_2(K)}^2.
 \end{gather}
\end{subequations}
The first is the standard Friedrichs inequality while the second is an
analogue for broken functions proved
in~\cite[Lemma~4.1]{DemkoGopal11a}.  Next, we need the
simple observation  that the kernel of $\Ba$ is parameterized by
$E(q,u) := (q, u, \nabla u, \tr(u), \trn (q))$ as
\begin{equation}
  \label{eq:E-map-diffus}
  \ker(\Ba) = \{ E(q,u) : u \in \Ho^1(\om), \;
  q \in \Hdzero \},
\end{equation}
where $\Hdzero := \{ q \in \Hdiv : \dive q = 0\}$.  Indeed, by
\eqref{eq:b-alpha}, $\clz \in \ker(\Ba)$ if and only if
$b_\alpha((q,u,p,\hat u,\hat q_n), (r, w))=0$ for all $(r,w)\in\Va$.  
Taking $\clv=(0,w)$, forces $\dive q = 0$ element-wise and
$\hat q_n = \trn q$ to be single-valued on interfaces, while taking
$\clv=(r,0)$ forces $p=\nabla u$ and $\hat u = \tr(u)$ and 
$u\in\Ho^1(\om)$. (Consult \cite[Lemma~4.4]{CH} if more details are
needed.)  By~\eqref{eq:Friedrichs-std} and the quotient norm
definition~\eqref{eq:quotient-norms}, it is then immediate that
\begin{equation}
  \label{eq:cE-diffus}
  c_E \, \| E(q,u) \|^2_{\Ua}
  \le \| \nabla u \|^2_{L_2(\om)} + \kappa_1^{-2}\| q \|^2_{L_2(\om)},
  \qquad \text{for all } u \in \Ho^1(\om), \; q \in \Hdzero,
\end{equation}
where $c_E^{-1} = 2 + 2 c_P^{-1}$.

\begin{proposition}[Verification of Assumption~\ref{asm:MDPG}]
  \label{prop:mdpg-verif}
  In the setting of \eqref{eq:B-C-example} and~\eqref{eq:phi-asm}, set 
  $\kappa_1,\kappa_2, \mu$ as in~\eqref{eq:kappas-mu}.
  Then all conditions of Assumption~\ref{asm:MDPG} hold for $\Ba$ and
  $\Ca$ in \eqref{eq:B-C-example}.  Specifically, 
  \begin{enumerate}[label={\rm(\roman*)}]
  \item $\Ba: \Ua \to \Va'$ is surjective for each $\alpha\in\cl Q$, i.e.,    \eqref{eq:S4-B-surjective} holds;
     moreover, there is a $c_b > 0$,
    depending only on the $c_P$ in \eqref{eq:Friedrichs2},  such that 
    $\| \Ba' \clv \|_{\Uap}^2 \ge c_b \| \clv \|^2_{\Va}$ for all
    $\clv \in \Va$; 
  \item the inequality~\eqref{eq:S4-infsup} holds with $c_B = 1/c_b$; 

  \item the inequality~\eqref{eq:S4-mono} holds with
    $c_U = c_E c_r/4$ and
    \[
      c_V = c_B \Big[c_r +
      \frac 1 2  (\kappa_2/\kappa_1)
      +
      4 (\kappa_2/\kappa_1)^2  c_E^{-1} c_r^{-1}\Big],
    \]
      where
      $    r = \mu/\kappa_2^2$ 
      and 
      $c_r= \big(
      1 + r - \sqrt{ (1 + r)^2 - 3 r } \big)/2;
      $
  \item the implication~\eqref{eq:C-zero} holds, and 
  \item the inequality~\eqref{eq:S4-Lip} holds with
    $c_{\rm Lip} = \kappa_1^{-2} \big[ (\kappa_1^2 + \kappa_2^2)
    (\kappa_1^2 + \kappa_1) \big]^{1/2}$.
    
  \end{enumerate}
  All the constants involved, $c_B, c_U, c_V,$ and $c_{\rm Lip}$, are
  independent of the mesh~$\oh$.
\end{proposition}
\begin{proof}
  Writing $\clz = (q,u,p,\hat u,\hat q_n)$ and
  $\clz_i=(q_i,u_i,p_i,\hat u_i,\hat q_{n,i})$ for generic elements
  of~$\Ua$, and $\clv=(r,w)$ for those in~$\Va$, we give quick
  pointers to proofs of every item.

  (i) By definition, $\Ba' \clv \in \Uap$ acts by
  $(\Ba' \clv)(\clz) = b_\alpha(\clz,\clv) = \ip{\hat q_n, w}_{\partial\oh}
  - \ip{r\cdot n, \hat u}_{\partial\oh} + (p,r)_\oh + (u,\div r)_\oh
  - (q,\nabla w)_\oh$. Thus,   taking suprema over individual components of $\clz$, 
  \[
    \sup_{0\ne \clz \in \Ua} \frac{b_\alpha( \clz,\clv)^2}{\|\clz\|_{\Ua}^2}
    \ge \kappa_1^2 \left( \|\nabla w\|^2_{L_2} + \sup_{0 \ne \hat q_n \in H^{-1/2}(\doh)}
    \frac{\ip{ \hat q_n, w}^2}{ \| \hat q_n \|_{H^{-1/2}(\doh)}^2} \right) 
  + \|r\|^2_{\Hdivoh}.
  \]
  By~\eqref{eq:Friedrichs-generalized}, we conclude 
  that  $\| \Ba' \clv \|_{\Uap}^2 \ge c_b \|\clv\|_\Va^2$ with
  $c_b = \min(c_P,1)/2$.  (Consult \cite[Lemma~4.2]{CH} if more details are
  needed.)  Accordingly,  $\Ba'$ is injective with closed range, so $\Ba$ is
  surjective onto $\Vap$, i.e.,~\eqref{eq:S4-B-surjective} holds.

  (ii) By inf-sup duality, the boundedness of $\Ba'$ from below in (i) 
  implies the same for $\Ba$ restricted to $\ker(\Ba)^\perp$, which
  contains elements of the form $\clz - \PB \clz$, i.e.,
  \begin{equation}
    \label{eq:31-alt}
    c_b\, \| \clz - \PB \clz \|_{\Ua}^2
    \le \| \Ba (\clz - \PB \clz) \|_{\Vap}^2 = \| \Ba \clz \|_{\Vap}^2,
  \end{equation}
  which is~\eqref{eq:S4-infsup} with $c_B = 1/c_b$.  (Consult
  \cite[Lemma~4.3]{CH} if more details are needed.)

  (iii) Write the first three components of $\clz_1-\clz_2$ as
  ${q_{12}} = q_1-q_2$, $u_{12} = u_1 - u_2$, ${p_{12}} = p_1-p_2$,
  and let ${\sigma_{12}} = \sigma_\alpha(p_1) - \sigma_\alpha(p_2)$.
  Using \eqref{eq:E-map-diffus}, write
  $\PB(\clz_1-\clz_2) = E(q_0,u_0)$ with $q_0\in\Hdzero$ and
  $u_0\in\Ho^1(\om)$. The difference of tuples 
  $t_0 = ( \kappa_1^{-1} q_0, \nabla u_0) = ( \kappa_1^{-1} q_0, p_0)$
  and $t_\perp = (\kappa_1^{-1}( q_0 - q_{12}), p_0 - p_{12})$, both in $L_2(\om)^d \times L_2(\om)^d$,  equals
  $t_0 - t_\perp = (\kappa_1^{-1} q_{12}, p_{12})$.
  By~\eqref{eq:cE-diffus},
  $ c_E\|\PB(\clz_1-\clz_2)\|^2_{\Ua} \le \| t_0 \|_{L_2}^2 \le 2 \|
  t_0 - t_\perp \|_{L_2}^2 + 2 \| t_\perp \|_{L_2}^2.$ Bounding
  $t_\perp$ by~\eqref{eq:31-alt} (which is the already proved~\eqref{eq:S4-infsup}), 
  \begin{equation}
    \label{eq:cE-PBdelta}
    \frac{c_E}{2}P_{12}%
    \;\le\; \|{p_{12}}\|^2_{L_2}
    + \kappa_1^{-2}\|{q_{12}}\|^2_{L_2}
    +  c_B \cl B_{12}. %
  \end{equation}
  where $P_{12} = \|\PB(\clz_1-\clz_2)\|^2_{\Ua}$ and
  $\cl B_{12} =
  \|\Ba(\clz_1-\clz_2)\|^2_{\Vap}$. (Equation~\eqref{eq:cE-PBdelta} is
  comparable to \cite[Equation~(24)]{CH} but we are unable to obtain
  their constants.)
  
  We need to connect the bound~\eqref{eq:cE-PBdelta} to
  $\cl C_{12} := \big\langle \Ca(\clz_1) - \Ca(\clz_2), \,\clz_1-\clz_2
  \big\rangle$. Expanding the
  right hand side of~\eqref{eq:C-mdpg'}, yields
  \begin{align}
    \nonumber
    \cl C_{12}
    &=
      \kappa_1^{-2}
      \Big[ \kappa_2 (-{\sigma_{12}})\cdot{p_{12}} + \|{q_{12}}\|^2_{L_2}
      + \kappa_2({q_{12}},{p_{12}})_{L_2} - ({\sigma_{12}},{q_{12}})_{L_2}
      \Big]
    \\ \label{eq:211}
    & \ge  \kappa_1^{-2}
      \left[    \frac{\kappa_2^2}{\mu}\|{p_{12}}\|^2_{L_2}
      + \|{q_{12}}\|^2_{L_2}
      + \kappa_2({q_{12}},{p_{12}})_{L_2} 
      -\sqrt{\kappa_1}\|{p_{12}}\|_{L_2}\|{q_{12}}\|_{L_2}      
      \right]
  \end{align}
  using~\eqref{eq:globalized-sigma} and the value of $\kappa_2$
  in~\eqref{eq:kappa-alpha}. Expressing the cross term as
  $({q_{12}},{p_{12}})_{L_2} = (q_0, p_0)_{L_2} + (q_{12} - q_0,
  p_0)_{L_2} + (q_0, p_{12} - p_0) + (q_{12} - q_0, p_{12} - p_0)$, 
  note that the first term satisfies
  $(q_0, p_0)_{L_2} = (q_0, \nabla u_0)_{L_2} = -(\div q_0, u_0) = 0$.
  The last term can be bounded by the inequality of arithmetic and
  geometric means with weight $\kappa_1^{-1}$,
  \begin{align*}
    \frac{\kappa_2}{\kappa_1^{2}}\,|(q_{12} - q_0, p_{12} - p_0)_{L_2}|
    &\le
      \frac{\kappa_2}{2\kappa_1}\!\Big(
      \kappa_1^{-2}\|q_{12}-q_0\|_{L_2}^2 + \|p_{12}-p_0\|_{L_2}^2
      \Big)
    \le
      \frac{c_B\kappa_2}{2\kappa_1}\,\cl B_{12}
  \end{align*}
  by virtue of~\eqref{eq:S4-infsup}. For the remaining terms, we use
  Young's inequality, 
  \begin{align*}
    \frac{\kappa_2}{\kappa_1^{2}}\,|(q_{12} - q_0, p_0)_{L_2}|
    &\le
      \frac{\kappa_2}{2\veps_1}\,\kappa_1^{-2}\|q_{12}-q_0\|_{L_2}^2
      + \frac{\kappa_2\veps_1}{2}\,\kappa_1^{-2}\|p_0\|_{L_2}^2
      \le
      \frac{c_B\kappa_2}{2\veps_1}\,\cl B_{12}
      + \frac{\kappa_2\veps_1}{2\kappa_1^2}\,P_{12},
  \end{align*}
  using \eqref{eq:S4-infsup}, where~$\veps_1$ is a free
  parameter. Similarly using a second free parameter~$\veps_2$,
  \begin{align*}
    \frac{\kappa_2}{\kappa_1^{2}}\,|(q_0, p_{12} - p_0)_{L_2}|
    &\le
      \frac{\kappa_2\veps_2}{2}\,P_{12} %
      + \frac{c_B\kappa_2}{2\veps_2\kappa_1^2}\, 
      \cl B_{12}. %
  \end{align*}
  Summing these cross-term 
  contributions and setting
  \[
    \tilde{c}_B(\veps_1,\veps_2) = \frac 1 2 {c_B\kappa_2}
    \!\left( \kappa_1^{-1} + \veps_1^{-1} + \veps_2^{-1}\kappa_1^{-2}\right),
  \]
  the estimate of~\eqref{eq:211} implies
  \begin{align}
    \label{eq:171}
    \cl C_{12}
    +
    \tilde{c}_B(\veps_1,\veps_2)\,
    \cl B_{12} %
    +
    \left(\frac{\veps_1\kappa_2}{2\kappa_1^2} + \frac{\kappa_2\veps_2}{2}\right)
    P_{12} %
    \ge  T,
  \end{align}
  where
  $T = \kappa_1^{-2} \big[(\kappa_2^2/\mu)\|{p_{12}}\|^2_{L_2} +
  \|{q_{12}}\|^2_{L_2}
  -\sqrt{\kappa_1}\|{p_{12}}\|_{L_2}\|{q_{12}}\|_{L_2}\big].$

  We can  view $T$ as a quadratic form of the matrix $M(t)$ defined below:
  \[
    r = \frac{\mu}{\kappa_2^2},
    \quad
    t = \frac{ \mu \kappa_1}{\kappa_2^2},
    \quad 
    M(t) =
    \begin{bmatrix}
      1     & -\frac 1 2 \sqrt{r/t}
      \\
      -\frac 1 2 \sqrt{r/t} & r/t^2     
    \end{bmatrix},
    \quad
    y =
    \begin{bmatrix}
      \kappa_1^{-1}\| q_{12} \|_{L_2} \\
      \phantom{\kappa_1^{-1}}\| p_{12} \|_{L_2}
    \end{bmatrix}.
  \]
  Clearly $T = y \cdot M(t) y$. By \eqref{eq:20},
  $0 < \kappa_1 < \kappa_2^2/\mu$, so $ 0 < t < 1$. The smallest
  eigenvalue of $M(t)$, namely
  $\lambda_-(t) = \frac 1 2 \big[ 1 + r/t^2 - \sqrt{ (1 + r/t^2)^2 -
    r (4-t)/t^2 } \big],$ is monotonically decreasing in
  $0 < t \le 1$ (see Remark~\ref{rem:mon-detivative-ew} below), so 
  $T = y \cdot M(t) y \ge \lambda_-(1) \,y\cdot y$. Thus, letting
  $c_r = \lambda_-(1),$ the estimate of \eqref{eq:171} implies
  \begin{align*}
    c_r\Big( \| p_{12}\|_{L_2}^2  +  \kappa_1^{-2} \| q_{12}\|_{L_2}^2\Big)
    \le T \le 
    \cl C_{12}
    + \tilde{c}_B(\veps_1,\veps_2)\,
    \cl B_{12} %
    +
    \left(\frac{\veps_1\kappa_2}{2\kappa_1^2} + \frac{\kappa_2\veps_2}{2}\right)
    P_{12}.
  \end{align*}
  Returning to \eqref{eq:cE-PBdelta}, this gives 
  \begin{align}
    \label{eq:17}
    \left(\frac{c_Ec_r}{2} - \frac{\veps_1\kappa_2}{2\kappa_1^2}
    - \frac{\kappa_2\veps_2}{2}\right)
    P_{12} %
    \;\le\;
    \cl C_{12}
    + \Big(c_Bc_r + \tilde{c}_B(\veps_1,\veps_2)\Big)
    \, \cl B_{12}.
  \end{align}
  Choosing $\veps_1 = \frac{1}{4} c_Ec_r\kappa_1^2/\kappa_2$ and
  $\veps_2 = \frac{1}{4}c_Ec_r / \kappa_2,$ the left hand side becomes
  $\frac 1 4 {c_Ec_r} P_{12}$ and we obtain~\eqref{eq:S4-mono}.  (A similar
  estimate with different constants is stated
  in~\cite[Lemma~4.6]{CH}.)

  (iv) Suppose $\clz=(q,u,p,\hat u,\hat q_n) \in \Ua$ satisfies
  $\langle \Ca(\clz), \clz_0 \rangle = 0$ for all
  $\clz_0 \in \ker(\Ba)$.  Set $g := q - \sigma_\alpha(p) \in
  L_2(\om)^d$.  Using~\eqref{eq:E-map-diffus}
  and~\eqref{eq:C-mdpg'}, for every $u_0 \in \Ho^1(\om)$ and
  $q_0 \in \Hdzero$,
  \[
    0 = \big\langle \Ca(\clz), E(q_0, u_0) \big\rangle
    = \kappa_1^{-2} \big( g, \; q_0 + \kappa_2 \nabla u_0 \big)_\oh.
  \]
  Choosing $q_0 = 0$ gives $g \perp \nabla \Ho^1(\om)$, and choosing
  $u_0 = 0$ gives $g \perp \Hdzero$.  Since
  $\nabla \Ho^1(\om) \oplus \Hdzero$ equals  all 
  $L_2(\om)^d$, by the Helmholtz decomposition,  we conclude that 
  $g = 0$, and hence
  $\Ca(\clz) = 0$ in $\Uap$, which is~\eqref{eq:C-zero}.

  (v) For $\clz_1,\clz_2 \in \Ua$ and any $\clz \in \Ua$, the
  Cauchy--Schwarz inequality and the global Lipschitz bound
  $\|\sigma_\alpha(p_1)-\sigma_\alpha(p_2)\|_{L_2} \le \sqrt{\kappa_1}
  \, \|p_1-p_2\|_{L_2}$ (that follows from~\eqref{eq:sigma-mono-lip}),
  give
  \[
    \big| \langle \Ca(\clz_1) - \Ca(\clz_2), \clz \rangle \big|
    = \kappa_1^{-2}
    \big| \big( {q_{12}}-{\sigma_{12}}, \, \kappa_2 p + q
    \big)_{L_2} \big|
    \le \kappa_1^{-2}
    \big( \sqrt{\kappa_1}\,\|{p_{12}}\|_{L_2} + \|{q_{12}}\|_{L_2} \big)
    \, \| \kappa_2 p + q \|_{L_2}.
  \]
  Both factors are now expressed in the weighted
  norm~\eqref{eq:mdpg-norms}. By the Cauchy--Schwarz inequality
  in~$\bb R^2$,
  \[
    \sqrt{\kappa_1}\,\|p_{12}\|_{L_2} + \|q_{12}\|_{L_2}
    = \sqrt{\kappa_1}\,\|p_{12}\|_{L_2} + \kappa_1\,\kappa_1^{-1}\|q_{12}\|_{L_2}
    \le \big(\kappa_1 + \kappa_1^2\big)^{1/2}
    \|\clz_1 - \clz_2\|_{\Ua},
  \]
  and similarly
  $\|\kappa_2 p + q\|_{L_2} \le \kappa_2\|p\|_{L_2} + \kappa_1\,\kappa_1^{-1}\|q\|_{L_2}
  \le (\kappa_1^2 + \kappa_2^2)^{1/2}\|\clz\|_{\Ua}$. Combining these
  bounds yields~\eqref{eq:S4-Lip} with the stated $c_{\rm Lip}$.
\end{proof}

\begin{remark}
  \label{rem:mon-detivative-ew}
  The tedium of verifying the monotonicity of $\lambda_-(t)$ mentioned
  in the above proof is alleviated  by substituting
  $\gamma = r/t^2$, expressing
  $ \lambda_-(\gamma) = \tfrac{1}{2}\bigl[(1+\gamma) -
  \sqrt{(1-\gamma)^2+\sqrt{r\gamma}}\bigr],$
  and differentiating with
  respect to $\gamma$. Then simplifications show that
  $d\lambda_-/ d\gamma >0$ when $r < \gamma$ (which is the case 
  as $t<1$).
\end{remark}

\begin{remark}
  The constant $c_r$ in Proposition~\ref{prop:mdpg-verif}, being the
  smallest eigenvalue of $M(1)$, satisfies
  \[
    \frac{3r}{4(r+1)} \;\le\; c_r \;\le\; \frac{1+r}{2}.
  \]  
  The upper bound is half the trace of $M(1)$ and the lower bound is
  the determinant-to-trace ratio of $M(1)$.
\end{remark}

\begin{remark}[Verification for the $\rho$-formulation]
  \label{rem:mdpg-rho-verification}
  For the alternative MDPG formulation described in
  Remark~\ref{rem:mdpg-rho-formulation}, the verification of
  Assumption~\ref{asm:MDPG} proceeds by an entirely analogous argument.
  Items (i) and (ii) are unchanged, as they concern only the operator
  $\Ba$, which is identical in both formulations. For items (iii)--(v),
  one uses the bi-Lipschitz bounds~\eqref{eq:rho-bilip} for $\rho_\alpha$
  in place of the monotonicity and Lipschitz bounds~\eqref{eq:sigma-mono-lip}
  for $\sigma_\alpha$. The proof structure remains the same: the closure
  difference becomes $\rho_\alpha(q_1) - \rho_\alpha(q_2)$ in place of
  $\sigma_\alpha(p_1) - \sigma_\alpha(p_2)$, and the monotonicity estimate
  in~\eqref{eq:211} uses the inner product
  $-(\rho_\alpha(q_1) - \rho_\alpha(q_2)) \cdot (q_1 - q_2)$ with the
  lower bound from~\eqref{eq:rho-bilip}. Since both $\sigma_\alpha$ and
  $\rho_\alpha$ satisfy bilipschitz bounds with the same constants
  (inverted) derived from $\beta_1, \beta_2, \alpha_1, \alpha_2$, the
  resulting constants $c_U, c_V, c_{\rm Lip}$ are identical for both
  formulations.
\end{remark}

\begin{proposition}[Verification of Assumption~\ref{asm:fortin-b}]
  \label{prop:fortin-b-verif}
  In the setting of \eqref{eq:B-C-example} and~\eqref{eq:phi-asm}---with the
  discrete spaces indexed by degree $k\ge 0$ as in
  \eqref{eq:mdpg-discrete}---there is a
  continuous linear operator $\fpb: \Va \to \Vah$
  satisfying~\eqref{eq:Fortin-b}, and
  \begin{equation}
    \label{eq:const-preserve-Pi}
    (v_0, \, w - w_\varPi)_{L_2} = 0
    \qquad \text{ for all } (r, w) \in \Va,\; v_0 \in P_0,
  \end{equation}
  where $w_\varPi$ denotes the second component of $\fpb(r,w)$.
  The norm $\|\fpb\|_{L(\Va)}$ is bounded 
  independently of the mesh $\oh$ and of $\alpha \in \cl Q$.
\end{proposition}
\begin{proof}
  This immediately follows from the results of~\cite{GopalQiu14} or
  the more recent review~\cite{DPGacta}: take
  $\fpb (r, v) = (\varPi^{\mathrm{div}}_{k+2} r,
  \varPi^{\mathrm{grad}}_{k+d} v)$ for any $(r, v) \in \Va$ using the
  operators $\varPi^{\mathrm{grad}}_{k+d}$
  of~\cite[Theorem~5.4]{DPGacta} and $\varPi^{\mathrm{div}}_{k+2}$
  of~\cite[Theorem~5.8]{DPGacta}. By construction~\cite{GopalQiu14},
  $\varPi^{\mathrm{grad}}_{k+d}$ preserves the moments of degree up to
  $k$ on each element, in particular the element averages,
  which gives~\eqref{eq:const-preserve-Pi}.  The arguments there are repeatable
  here since the bilinear form $\ba(\clz,\clv) = (\Ba \clz)(\clv)$ is
  the ultraweak DPG form for the linear system $\dive q = f$,
  $p = \nabla u$ (which differs from prior standard studies only in
  that the two equations are decoupled, a difference immaterial
  here). The cited results provide mesh-independent bounds for the
  norms of $\varPi^{\mathrm{div}}_{k+2}$ and
  $\varPi^{\mathrm{grad}}_{k+d}$, and hence for~$\fpb$.  Additionally,
  since $\Ba$ carries no $\alpha$-dependence, it is also independent
  of $\alpha \in \cl Q$.
\end{proof}

\begin{theorem}
  \label{thm:vc-MDPG-example}
  In the setting of \eqref{eq:B-C-example} and~\eqref{eq:phi-asm}, set
  $\kappa_1,\kappa_2, \mu$ as in~\eqref{eq:kappas-mu}.  Then the error
  estimator $\cl L(\alpha,{\clw_h})$, given by the MDPG loss function in \eqref{eq:C&H-lossh}, is globally
  reliable and efficient for any trial function $\clw_h$ in $\Uah$,
  i.e., \eqref{eq:reliability-L} and \eqref{eq:efficiency-L} hold with
  mesh-independent constants. If the right hand side $f$ of
  \eqref{eq:4} is a piecewise constant function in $P_0$, then there
  are mesh-independent constants $c, C$ such that
  \begin{equation}
    \label{eq:18}
    c \| \clu  - \clw_h \|_\Ua^2  \le
    \cl L(\alpha,\clw_h) \le C
    \| \clu - \clw_h \|_\Ua^2,
  \end{equation}
  for all $\clw_h \in \Uah$.
\end{theorem}
\begin{proof}
  This follows from Theorem~\ref{thm:gen-nonlinear-dpg} since
  Propositions~\ref{prop:mdpg-verif} and~\ref{prop:fortin-b-verif}
  verify its conditions---the MDPG choice
  $\Ma{\clz} = \| \Ca(\clz) \|_{\Uap}$ satisfies
  Assumption~\ref{asm:N} with $c_M = c_{\rm Lip}$, as noted following
  Assumption~\ref{asm:N}. In addition, 
  when $f \in P_0$, we conclude that $\ell ( (r, w) - \fpb (r,w)) = (f, w - \varPi^{\mathrm{grad}}_{k+d} w)_{L_2} = 0$ due
  to \eqref{eq:const-preserve-Pi}, so $\osc(\ell)=0$, and 
  \eqref{eq:reliability-L}--\eqref{eq:efficiency-L} give~\eqref{eq:18}.
\end{proof}

Note that the MDPG variational correctness result of
Theorem~\ref{thm:vc-MDPG-example} is global, without any assumptions on the
specific form of $\alpha(x)$,   in the sense that it does not restrict trial functions
$\clw_h$ to any small neighborhood of the exact solution, and
allows for arbitrary degree $(k)$ discretizations (although the norm
$\| \fpb\|_{L(\Va)}$ may depend on $k$).

\subsection{Variational correctness of the MFOSLS loss function}
\label{ssec:mfosls-verif}

We now verify the assumptions in the MFOSLS setting
of~\S\ref{sec:fosls-loss-function}, keeping the nonlinearity
of~\eqref{eq:phi-asm} and setting $\kappa_1, \kappa_2, \mu$ as
in~\eqref{eq:kappas-mu}.  Since $\Vah = \Va$ and $\fpb = \mathrm{id}$,
Assumption~\ref{asm:fortin-b} is obviously  verified, so it remains to
verify Assumptions~\ref{asm:MDPG} and~\ref{asm:N}.  By
\eqref{eq:mfosls-b}, the kernel of $\Ba$ is now parameterized by
$E(q, u) := (q, u, \nabla u)$ as
\begin{equation}
  \label{eq:E-map-fosls}
  \ker(\Ba) = \{ E(q, u) : u \in \Ho^1(\om), \; q \in \Hdzero \},
\end{equation}
since $\Ba \clz = 0$ if and only if $\dive q = 0$ and
$p = \nabla u$.  Moreover, for $u \in \Ho^1(\om)$ and
$q \in \Hdzero$, the norm in~\eqref{eq:mfosls-norms} gives
$\| E(q, u) \|_{\Ua}^2 = \| u \|_{L_2}^2 + 2 \| \nabla u \|_{L_2}^2 +
\kappa_1^{-2} \| q \|_{L_2}^2$, so the Friedrichs
inequality~\eqref{eq:Friedrichs-std} shows that
\begin{equation}
  \label{eq:cE-fosls}
  c_E \, \| E(q,u) \|^2_{\Ua}
  \le \| \nabla u \|^2_{L_2(\om)} + \kappa_1^{-2}\| q \|^2_{L_2(\om)},
  \qquad u \in \Ho^1(\om), \; q \in \Hdzero,
\end{equation}
holds with $c_E^{-1} = 2 + c_P^{-1}$.

\begin{proposition}[Verification for MFOSLS]
  \label{prop:mfosls-verif}
  In the setting of~\eqref{eq:mfosls-setting} and~\eqref{eq:phi-asm},
  set $\kappa_1, \kappa_2, \mu$ as in~\eqref{eq:kappas-mu}.  Then all
  conditions of Assumption~\ref{asm:MDPG} hold for $\Ba$ and $\Ca$
  there, and Assumption~\ref{asm:N} holds for  $\Ma\cdot$, defined in
   ~\eqref{eq:mfosls-N}.  Specifically,
  \begin{enumerate}[label={\rm(\roman*)}]
  \item the surjectivity of~\eqref{eq:S4-B-surjective} holds;
  \item the inequality~\eqref{eq:S4-infsup} holds with
    $c_B = 1 + c_P^{-1}$, where $c_P$ is the Friedrichs constant
    of~\eqref{eq:Friedrichs-std};
  \item the inequality~\eqref{eq:S4-mono} holds with
    $c_U = c_E c_r / 4$ and $c_V$ given by the formula in
    Proposition~\ref{prop:mdpg-verif}(iii), using the above values of
    $c_E$ and $c_B$;
  \item the implication~\eqref{eq:C-zero} holds, and
  \item the inequality~\eqref{eq:S4-Lip} and both conditions of
    Assumption~\ref{asm:N} are valid with
    $c_{\rm Lip} = c_M =
    \kappa_1^{-2} \big[ (\kappa_1^2 + \kappa_2^2)
    (\kappa_1^2 + \kappa_1) \big]^{1/2}$.
  \end{enumerate}
\end{proposition}
\begin{proof}
  Writing $\clz = (q, u, p)$ for a generic element of $\Ua$, we prove
  each item.

  (i)--(ii) Given $\clz \in \Ua$, let $\varphi \in \Ho^1(\om)$ solve
  the Dirichlet problem
  $(\nabla \varphi, \nabla v)_{L_2} = -(\dive q, v)_{L_2}$ for all
  $v \in \Ho^1(\om)$.  Then $\nabla \varphi \in \Hdiv$ with
  $\dive \nabla\varphi = \dive q$, and testing with $v = \varphi$, the
  Friedrichs inequality~\eqref{eq:Friedrichs-std} gives
  $\| \nabla \varphi \|_{L_2} \le c_P^{-1/2} \| \dive q \|_{L_2}$.  By
  \eqref{eq:E-map-fosls}, the function
  $\clz_0 := (q - \nabla\varphi, \, u, \, \nabla u)$ belongs to
  $\ker(\Ba)$, and $\clz - \clz_0 = (\nabla\varphi, 0, p - \nabla u)$. This gives
  \[
    \| \clz - \PB\clz \|_{\Ua}^2
    \le \| \clz - \clz_0 \|_{\Ua}^2
    = \kappa_1^{-2} \big( \| \nabla \varphi \|_{L_2}^2
    + \| \dive q \|_{L_2}^2 \big) + \| p - \nabla u \|_{L_2}^2
    \le ( 1 + c_P^{-1} ) \, \| \Ba \clz \|_{\Vap}^2,
  \]
  where the last step uses
  $\| \Ba \clz \|_{\Vap}^2 = \kappa_1^{-2} \| \dive q \|_{L_2}^2 + \|
  p - \nabla u \|_{L_2}^2$, a consequence of the self-duality of the
  weighted $L_2$ norms in~\eqref{eq:mfosls-norms}.  This
  proves~(ii).  Regarding~(i), given any $(g_1, g_2) \in L_2(\om) \times L_2(\om)^d
  \cong \Vap$, solving the same Dirichlet problem with $-\dive q$
  replaced by $-g_1$ and setting $\clz = (\nabla\varphi, 0, g_2)$,
  yields $\Ba \clz = (g_1, g_2)$.

  (iii) The proof of Proposition~\ref{prop:mdpg-verif}(iii) applies
  verbatim.  Indeed, that argument uses only the
  already-established~(ii), the estimate~\eqref{eq:cE-fosls} in place
  of~\eqref{eq:cE-diffus}, the orthogonality
  $(q_0, p_0)_{L_2} = (q_0, \nabla u_0)_{L_2} = 0$ (valid here since
  $\dive q_0 = 0$, and $u_0 \in \Ho^1(\om)$), as well as manipulations of the
  $L_2$-components $q_{12}, p_{12}$ that are insensitive to the
  remaining components of the trial tuples.

  As for (iv), suppose $\clz \in \Ua$ satisfies
  $\ip{\Ca(\clz), \clz_0} = 0$ for all $\clz_0 \in \ker(\Ba)$.  Set
  $g := q - \sigma_\alpha(p) \in L_2(\om)^d$.
  Using~\eqref{eq:E-map-fosls} and~\eqref{eq:C-mdpg'}, for every
  $u_0 \in \Ho^1(\om)$ and $q_0 \in \Hdzero$, we have 
  $    0 = \big\langle \Ca(\clz), E(q_0, u_0) \big\rangle
    = \kappa_1^{-2} \big( g, \; q_0 + \kappa_2 \nabla u_0 \big)_\oh.
  $
  Choosing $q_0 = 0$, gives $g \perp \nabla \Ho^1(\om)$, and choosing
  $u_0 = 0$ gives $g \perp \Hdzero$.  Since
  $\nabla \Ho^1(\om) \oplus \Hdzero$ equals all of $L_2(\om)^d$, by the Helmholtz
  decomposition, $g = 0$, and hence $\Ca(\clz) = 0$ in $\Uap$.

  To see (v), as in the proof of Proposition~\ref{prop:mdpg-verif}(v),  the global Lipschitz bound
  $\|\sigma_\alpha(p_1)-\sigma_\alpha(p_2)\|_{L_2} \le \sqrt{\kappa_1}
  \, \|p_1-p_2\|_{L_2}$ gives for any
  $\clz_1, \clz_2, \clz \in \Ua$
  \begin{equation}
    \label{eq:C12-mfosls}
    \big| \langle \Ca(\clz_1) - \Ca(\clz_2), \clz \rangle \big|
    \le \kappa_1^{-2}
    \big( \sqrt{\kappa_1}\,\|{p_{12}}\|_{L_2} + \|{q_{12}}\|_{L_2} \big)
    \, \| q + \kappa_2 p \|_{L_2}.
  \end{equation}
  In view the norm~\eqref{eq:mfosls-norms},   the
  Cauchy--Schwarz inequality in $\bb R^2$ provides
  $
      \sqrt{\kappa_1}\, \| p_{12} \|_{L_2} + \| q_{12} \|_{L_2}
    \le \big( \kappa_1 + \kappa_1^2 \big)^{1/2}
    \big( \| p_{12} \|_{L_2}^2
    + \kappa_1^{-2} \| q_{12} \|_{L_2}^2 \big)^{1/2}.
    $
    Hence 
    \begin{gather}
      \label{eq:mfosls-CS2}
      \sqrt{\kappa_1}\, \| p_{12} \|_{L_2} + \| q_{12} \|_{L_2}
      \le \big( \kappa_1 + \kappa_1^2 \big)^{1/2}
      \| \clz_1 - \clz_2 \|_{\Ua}.
    \end{gather}
    Similarly, one obtains
  \begin{gather}
    \label{eq:mfosls-CS1}
    \| q + \kappa_2 p \|_{L_2}
    \le \kappa_1 \, \kappa_1^{-1} \| q \|_{L_2}
    + \kappa_2 \| p \|_{L_2}
    \le \big( \kappa_1^2 + \kappa_2^2 \big)^{1/2} \| \clz \|_{\Ua}.
  \end{gather}
  Combining~\eqref{eq:C12-mfosls},~\eqref{eq:mfosls-CS2} and~\eqref{eq:mfosls-CS1},  proves~\eqref{eq:S4-Lip} with the
  stated constant $c_{\rm Lip}$.  Condition~\eqref{eq:N-pairing} also follows
  from~\eqref{eq:mfosls-CS1}, since
  $\ip{ \Ca(\clz_1), \clz }
  = \kappa_1^{-2} ( q_1 - \sigma_\alpha(p_1),
  \, q + \kappa_2 p )_{L_2}
  \le \Ma{\clz_1} \| \clz \|_{\Ua}$,
  by~\eqref{eq:mfosls-N}.  Finally, if $\bar\clz \in \Ua$ satisfies
  $\Ca(\bar\clz) = 0$, then testing with tuples $\clz = (0, 0, p)$
  shows $(\bar q - \sigma_\alpha(\bar p), p)_{L_2} = 0$ for all
  $p \in L_2^d$, i.e., $\bar q = \sigma_\alpha(\bar p)$.  Hence, for
  any $\clz_1 \in \Ua$, writing
  $q_1 - \sigma_\alpha(p_1) = (q_1 - \bar q) -
  (\sigma_\alpha(p_1) - \sigma_\alpha(\bar p))$, the Lipschitz bound
  and~\eqref{eq:mfosls-CS2} (with $\clz_2 = \bar\clz$) give
  \[
    \Ma{\clz_1}
    \le \kappa_1^{-2} \big( \kappa_1^2 + \kappa_2^2 \big)^{1/2}
    \big( \| q_1 - \bar q \|_{L_2}
    + \sqrt{\kappa_1} \| p_1 - \bar p \|_{L_2} \big)
    \le c_M \| \clz_1 - \bar\clz \|_{\Ua},
  \]
  which is~\eqref{eq:N-lip} with $c_M = c_{\rm Lip}$ as stated.
\end{proof}

\begin{theorem}
  \label{thm:vc-MFOSLS-example}
  In the setting of~\eqref{eq:mfosls-setting} and~\eqref{eq:phi-asm},
  set $\kappa_1, \kappa_2, \mu$ as in~\eqref{eq:kappas-mu} and
  $\kappa$ as in~\eqref{eq:S4-kappa}.  Then the MFOSLS loss
  function~\eqref{eq:mfosls-loss} is globally reliable and efficient
  for any element of the full trial  space $\Ua =
  \Hdiv \times \Ho^1(\om) \times L_2(\om)^d$. (Indeed no finite elements or mesh was used.) I.e., there  are constants $c, C > 0$ such that
  \begin{equation}
    \label{eq:mfosls-vc}
    c \, \| \clu - \clz \|_{\Ua}^2
    \le \cl L(\alpha, \clz)
    \le C \, \| \clu - \clz \|_{\Ua}^2,
    \qquad \text{ for all } \clz \in \Ua, \; \alpha \in \cQ.
  \end{equation}
\end{theorem}
\begin{proof}
  By Proposition~\ref{prop:mfosls-verif}, Assumptions~\ref{asm:MDPG}
  and~\ref{asm:N} hold.  Choosing $\Uah = \Ua$, $\Vah = \Va$, and
  $\fpb = \mathrm{id}$, Assumption~\ref{asm:fortin-b} holds trivially
  with $\| \fpb \|_{L(\Va)} = 1$ and $\osc(\ell) = 0$, as noted
  in~\S\ref{ssec:mfosls-abstract}.
  Theorem~\ref{thm:gen-nonlinear-dpg} then
  gives~\eqref{eq:mfosls-vc} with the $c$ of~\eqref{eq:8} and
  $C = \kappa \| \Ba \|_{L(\Ua, \Vap)}^2 + c_M^2$. Note that
  $\| \Ba \|_{L(\Ua,\Vap)}^2 \le 2$ by~\eqref{eq:mfosls-norms}.
\end{proof}

We conclude this section with a few remarks comparing MFOSLS and MDPG
loss functions. In contrast to Theorem~\ref{thm:vc-MDPG-example}, no
data oscillation restriction (like $f \in P_0$) appears in
Theorem~\ref{thm:vc-MFOSLS-example}, no mesh of finite elements
enters, and the two-sided
bound~\eqref{eq:mfosls-vc} holds on the entire trial space $\Ua$
rather than on a discrete subspace.  What MFOSLS requires instead is
conformity of the trial candidates, i.e., $q \in \Hdiv$ and
$u \in \Ho^1(\om)$ with exactly imposed boundary conditions.

\section{Application to neural network approximations}
\label{sec:nn}
\subsection{The learning problem}
\label{ssec:learning-problem}
We next use the above findings to ``learn'' the mapping $\alpha\mapsto \clu(\alpha)$.
As the term ``learn'' suggests we employ some sort of {\em regression} over a suitable
{\em hypothesis class} $\mathcal H= \mathcal H(\Theta)\subset \Ua$, $\alpha\in \cQ$,
of functions with input variables $x\in\Omega, \alpha\in \cQ$, determined by an array 
$\theta\in\Theta$ of trainable weights, postponing the specification of 
$\mathcal H(\Theta)$ for a moment.
Moreover,  assuming that $\cQ$ is endowed with a probability measure $\nu$, we view 
the exact map $\alpha\mapsto \clu(\alpha)$ as the minimizer of the population risk
\newcommand{\Y}{\mathbb{Y}}
\newcommand{\X}{\mathbb{X}}
\begin{equation}
\label{minexp}
\clu =\argmin_{\clz \in \X}\mathbb E_{\alpha\sim\nu}\big[\cl L(\alpha,\clz(\alpha))\big],
\end{equation}
where
$$
\X:= L_{2,\nu}(\cQ;(\Ua)_{\alpha\in\cQ}),
$$
understood in the sense of direct integrals~\cite{BDO},
accommodates the solutions of \eqref{eq:Au=ell} as functions of $x\in \Omega$ and 
$\alpha\in \cQ$. In these terms, the accuracy of an approximation 
$
\alpha\mapsto \clu(\alpha;\theta^*)  
$ to the solution operator
\begin{equation}
    \label{solop}
\mathcal F:   \alpha \mapsto \clu(\cdot;\alpha)
\end{equation}
is measured through 
$$
\|\clu-\clu(\theta^*)\|^2_\X = \mathbb{E}_{\alpha\sim\nu}\big[\|\clu(\alpha)- \clu(\alpha;\theta^*)\|_\Ua^2\big].
$$
Variational correctness then means that 
\begin{equation}
    \label{vc}
c\,\|\clu-\clu(\theta^*)\|^2_\X
\le \mathbb E_{\alpha\sim\nu}\big[\cl L(\alpha,\clu(\alpha;\theta^*))\big]
\le C\,\|\clu-\clu(\theta^*)\|^2_\X
\end{equation}
with constants $0<c\le C$ independent of $\alpha$; by
Theorem~\ref{thm:gen-nonlinear-dpg}, this holds, e.g., whenever
$\osc(\ell)=0$.

Of course, one cannot exactly minimize the expectations, and a common approach in 
operator learning is to approximate $\clu:\Omega\times \cQ \to \mathbb R$ by 
minimizing instead an empirical counterpart. This requires computing high fidelity 
solutions of \eqref{eq:Au=ell} for some (randomly drawn) parameter samples 
$\alpha\in \cQ$. Given the slow decay of Monte Carlo rates,  achieving in the end 
an acceptable accuracy, requires a rather large number of such training samples, 
which especially for nonlinear PDEs may entail a prohibitive cost. Therefore, in 
view of \eqref{vc}, we instead minimize an empirical counterpart of the
loss itself, which requires no solution data:
\begin{equation}
    \label{eq:erm}
\clu(\alpha;\theta^*)\approx \clu(\alpha),\qquad \theta^*\in \argmin_{\theta\in\Theta}\frac{1}{\#\widehat\cQ_{\rm train}}\sum_{\alpha\in\widehat\cQ_{\rm train}} \cl L(\alpha, \clu(\alpha;\theta)),
\end{equation}
where $\widehat\cQ_{\rm train}\subset\cQ$ is a finite training set of independent samples drawn from 
$\cQ$ with distribution~$\nu$.

We primarily focus on  a finite-dimensional parameter set for
$\alpha$ by subdividing $\om$ into finitely many subdomains $\om_i$,
$i=1, \dots, m_\alpha$, and requiring $\alpha$ to be a constant in each
$\om_i$. Namely, using $\chi_i(x)$ to denote the indicator function of
$\om_i$, we define the parameter domain to be
\begin{equation}
  \label{eq:Qdefn}
  \mathcal{Q} = \left \{\sum_{i=1}^{m_\alpha} \alpha_i \chi_i(x) \; : \; \alpha_i > 0 \right \}.
\end{equation}
This piecewise-constant setting provides transparent nonlinear interface
problems while retaining a nontrivial finite-dimensional parameter-to-solution
map.  The compact coefficient ranges used in the numerical experiments are
specified below. This relatively simple setting suffices to highlight the issues discussed below.

We now describe the hypothesis classes used to approximate $\cl F$. Let $\Uah$ 
denote a hierarchy of $\Ua$-conforming finite element spaces on corresponding 
shape regular meshes of maximal element diameter~$h$, 
e.g.,~\eqref{eq:mdpg-discrete} for the MDPG loss and \eqref{eq:mfosls-discrete} 
for the MFOSLS loss. We work with so called hybrid representation formats where 
the functions used for prediction are piecewise polynomials as functions of the 
spatial variables whose coefficients are parameter dependent and represented as 
neural networks. Two major advantages concern the realization of essential 
boundary conditions and the ease of generating spatial derivatives while 
backpropagation remains reserved to parametric variables.

Specifically, 
suppressing the dependence on $h$ in the notation, suppose that we are given a basis 
$\psi_j$, $j=1, \dots, m_h$, for $\Uah$ and that an approximation of
$\cl F(\alpha)(x) =  \clu(x,\alpha) $ has been found, for each given $\alpha$ in
$\cl Q$, in their span
\[
  \Uah = \left\{
    \sum_{j=1}^{m_h} v_j\, \psi_j(x) : v_j \in \bb R \right\}
\]
by some method.  Then a computable approximation of the exact $\cl F$
in~\eqref{solop}, which we denote by $\cl F_h$, 
can be implemented as a 
map between Euclidean  vector spaces,
$\mathbb{F}: \mathbb{R}^{m_\alpha} \to \mathbb{R}^{m_h}$. The mapping
of functions $\cl F_h$ is generated from the mapping $\mathbb F$ of 
vectors by
\begin{equation}
  \label{eq:fun-map}
  \cl F_h\Big( \sum_{i=1}^{m_\alpha} \alpha_i \chi_i(x)\Big)
  = \sum_{j=1}^{m_h} \,[\mathbb F(\bbalpha)]_j\, \psi_j(x),
  \qquad x \in \om,  
\end{equation}
where $\bbalpha$ is the vector of coefficients of $\alpha \in \cl Q$, i.e.,
$\bbalpha = (\alpha_1, \alpha_2, \dots)$.
We shall construct such approximate mappings $\Fh$ using neural
networks. 

To construct the map $\mathbb{F}$ using an NN, we follow \cite{CDG} and
use an architecture proposed in~\cite{BDO},
described next. Let $\varrho$ denote a leaky ReLU function
defined componentwise by
\[
  \varrho: \mathbb R^N \to \mathbb R^N,\qquad
  \varrho(y)_i =\max( y_i, 10^{-3} y_i), \qquad y \in \bb R^N,
\]
for some integer $N$.
Next, we select another positive integer $r$ (often $r\ll N$ in low-rank
variants, as suggested in \cite{BDO}) and
consider 
matrices
$A_{\text{out}} \in \bb R^{m_h \times N},
A_m \in \bb R^{N \times r},
W_m \in \bb R^{r \times N},
A_{\text{in}} \in \bb{R}^{N \times {m_\alpha}},$
and vectors
$b_m \in \bb R^r,
b_{\text{out}} \in \bb{R}^{{m_h}},
b_{\text{in}} \in \bb{R}^{N}. 
$
For $m=1, \dots, l$, 
let
\begin{align*}
  & \Phi_m: \mathbb R^N \to \mathbb R^N, 
  && \Phi_m (z) =
     z %
     + A_m %
     \varrho\big( W_m%
     z %
     +
     b_m%
     \big),
  \\
  & L_{\text{out}}: \mathbb R^N \to \mathbb R^{m_h}, 
  && L_{\text{out}}(z) =
     A_{\text{out}}%
     z %
     +
     b_{\text{out}},
  \\
  & L_{\text{in}} : \mathbb R^{m_\alpha}\to \mathbb R^N,
  && L_{\text{in}}(p) =
     A_{\text{in}}
     p
     +
     b_{\text{in}},
\end{align*}
for any $z \in \bb R^N $ and $p \in \bb R^{{m_\alpha}}$. Then we approximate
  $\mathbb{F}$ by expressions of the form
  \begin{equation}
    \label{eq:NN-F}
  \mathbb{F}_\theta : \mathbb{R}^{m_\alpha} \to  \mathbb{R}^{m_h}, \qquad
  \mathbb{F}_\theta = L_{\text{out}} \circ \Phi_l \circ \dots \circ \Phi_1 \circ L_{\text{in}}
\end{equation}
where the subscript $\theta$ represents all the ``trainable'' degrees
of freedom (weights and biases) above, i.e.,
$  \theta = \{A_{\text{in}}, b_{\text{in}}, A_m, W_m, b_m,  b_{\text{out}}, A_{\text{out}}: \; m = 1, \dots, l\}.$
Let  $\Theta$ denote the set of all such values that  $\theta$ can take.
An appropriate  value of  $\theta$ in $\Theta$
is
found by a training process, as commonly done in the context of
all neural networks.
Let $\Fht : \cl Q \to \Uah$  denote the mapping of functions 
obtained by replacing $\bb F$ by $\bb F_\theta$ in~\eqref{eq:fun-map}.

\definecolor{depthA}{RGB}{244,109,67}
\definecolor{depthB}{RGB}{215,48,39}
\definecolor{depthC}{RGB}{153,15,21}
\definecolor{depthD}{RGB}{75,0,10}
\definecolor{depthAb}{RGB}{116,169,207}
\definecolor{depthBb}{RGB}{43,140,190}
\definecolor{depthCb}{RGB}{4,90,141}
\definecolor{depthDb}{RGB}{2,40,80}
\definecolor{epA}{RGB}{166,196,222}
\definecolor{epB}{RGB}{107,153,208}
\definecolor{epC}{RGB}{49,110,180}
\definecolor{epD}{RGB}{12,70,140}
\definecolor{epE}{RGB}{5,30,80}
\definecolor{degA}{RGB}{31,119,180}
\definecolor{degB}{RGB}{214,96,20}
\definecolor{degC}{RGB}{35,110,45}

\pgfplotsset{
  s6ablation/.style={
    width=8.7cm, height=5.9cm,
    xmajorgrids=true, ymajorgrids=true,
    grid style={gray!22, very thin},
    tick label style={font=\scriptsize},
    label style={font=\footnotesize},
    title style={font=\footnotesize\bfseries, yshift=-2pt},
    legend cell align=left,
    legend style={font=\scriptsize, draw=none, fill=white,
                  fill opacity=0.93, text opacity=1, inner sep=1.5pt,
                  row sep=-1pt},
  },
  s6best/.style={thick, mark size=1.5pt},
}

\section{Numerical studies}
\label{sec:numerics}

We compare MDPG and MFOSLS as losses for learning the parameter-to-solution
map. The experiments address four questions: how the loss scales with the
error, how the network depth affects the attainable accuracy, how the size
of the training set and the polynomial degree influence the accuracy
relative to that available in the trial space, and whether the approach
remains effective as the dimension of the parameter domain increases.
At the outset, we emphasize that our experiments are not meant to
validate the theory (which is already established by the proofs above), nor to
demonstrate large-scale capability (which requires dedicated
high-performance implementations outside the scope of this paper).
Rather, they illustrate how the method
responds to the key hyperparameters of the hypothesis classes
and %
how the error--loss relationship asserted by variational correctness
depends on the parameter range. 
All reference solutions are used only
for evaluation; training requires no solution data.

\subsection{Experimental setup}
\label{ssec:6.1}

We take $\om=(0,1)^2$, $f=1$, and the nonlinear profile
$\phi(t)=2+(1+t)^{-1}$ of Example~\ref{ex:invertible}, with
$\beta_1=2$ and $\beta_2=3$. For the loss--error and training comparisons,
we partition $\om$ into a $2\times2$ array of equal squares (see
Figure~\ref{fig:mixed-box-solutions}, top left). The coefficient
$\alpha$ is constant on each square, and its four subdomain values form
an $m_\alpha=4$ parameter vector sampled uniformly from
\begin{equation}
  \label{eq:alpha-boxes}
  \cl Q_n=J_n^4, \qquad J_n=[3/n,6/n], \qquad n=1,2,3.
\end{equation}
Only the final experiment, in \S\ref{ssec:mixed-box-experiments}, uses
a finer partition with $m_\alpha=16$ subdomain values and a different
parameter box specified there.

For each parameter box, we use the weights from the theory
of~\eqref{eq:kappa-alpha}, taking $\mu$ to be the smallest integer strictly
larger than $(\beta_2\alpha_2/(\beta_1\alpha_1))^2$. We choose $\kappa$
just above the theoretical threshold $\kappa_0:=c_V+c_Bc_U$, using
$\kappa=(1+2^{-20})\kappa_0$, with the constants for the respective
formulation. Thus, on $\cl Q_n$, $\mu=10$, $\kappa_1=n^2$, and
$\kappa_2=10n/3$. These boxes change the coefficient scale and the norm
weights while keeping the contrast $\alpha_2/\alpha_1=2$ fixed.

The networks~\eqref{eq:NN-F} predict coefficients in the finite element
spaces~\eqref{eq:mdpg-discrete} and~\eqref{eq:mfosls-discrete}.
We fix width $N=128$ and rank $r=32$, and use depth $l=5$ unless stated
otherwise. In every experiment, the training set $\widehat\cQ_{\rm train}$
of~\eqref{eq:erm} consists of independent uniform samples from the parameter
box used in that experiment; we denote its size by
$n_{\rm train}:=\#\widehat\cQ_{\rm train}$.
Training uses Adam with learning rate $10^{-4}$, batch size
$32$, and seed $42$; the test parameters are drawn with a
separate seed ($1730$), independently of the training set. Separately
sampled sets of $500$ validation
parameters and $1024$ test parameters are used for monitoring and
evaluation, respectively. Both formulations use the same test parameters.
The degree and training-sample comparisons below report the final
checkpoint after $20\,000$ epochs.

Let $\widehat\cQ$ denote this set of $1024$ test parameters. With
$\cl L$ denoting the selected MDPG or MFOSLS loss, write
\begin{align}
  \cl L_h(\clz;\widehat\cQ)
  &:=\frac{1}{\#\widehat\cQ}\sum_{\alpha\in\widehat\cQ}
       \cl L(\alpha,\clz(\alpha)),
       \label{emprisk}\\
  \cl E_h(\clw-\clz;\widehat\cQ)
  &:=\frac{1}{\#\widehat\cQ}\sum_{\alpha\in\widehat\cQ}
       \|\clw(\alpha)-\clz(\alpha)\|_{\Ua}^2.
       \label{empE}
\end{align}
The square root $\cl L_h^{1/2}$ is on the scale of the trial-norm
error.  The label FEM denotes Newton solutions of the discrete mixed
system~\eqref{eq:mixedC&H} in the same trial space as the network
output.  They provide (not the exact loss minimum but) a
spatial-discretization benchmark.

\subsection{Loss--error calibration}
\label{ssec:posteriori-ratio}

\begin{table}%
  \centering
  \small
  \setlength{\tabcolsep}{6pt}
  \caption{Test loss and loss--error ratios in the respective weighted
  trial norms, with $h=0.1$, $k=1$, $l=5$, $n_{\rm train}=4096$, after
  $2\times10^4$ epochs, on the same $1024$ test parameters as
  Tables~\ref{tab:ntrain-errors} and~\ref{tab:nn-errors}.}
  \label{tab:posteriori-ratio}
  \begin{tabular}{c|ccc|ccc}
    \hline
    & \multicolumn{3}{c|}{MDPG} & \multicolumn{3}{c}{MFOSLS}\\
    Box & $\cl L_h^{1/2}$ & $\bar r_{\rm est}$ & $K_{\rm est}$
        & $\cl L_h^{1/2}$ & $\bar r_{\rm est}$ & $K_{\rm est}$\\
    \hline
    $\cl Q_1$ & 0.442 & 50.9 & 24.70 & 0.195 & 135 & 8.49\\
    $\cl Q_2$ & 0.0179 & 1.24 & 8.09 & 0.0230 & 8.45 & 7.59\\
    $\cl Q_3$ & 0.0143 & 0.688 & 7.14 & 0.00752 & 0.269 & 21.78\\
    \hline
  \end{tabular}
\end{table}

At $h=0.1$, $k=1$, and $n_{\rm train}=4096$, we compare the loss with the
weighted trial-norm error using a reference of the same degree on a refined mesh,
$\clu_{h^\star}$, $h^\star=0.05$:
\begin{equation}
  \label{eq:rbar-est}
  \bar r_{\rm est}
  :=\frac{\cl L_h(\clz;\widehat\cQ)}
           {\cl E_h(\clu_{h^\star}-\clz;\widehat\cQ)}.
\end{equation}
Ideally, one would form this ratio with the exact states $\clu(\alpha)$ in
place of $\clu_{h^\star}(\alpha)$, but these are not available. To assess the
effect of this replacement, we let $\widehat\cQ_0\subset\widehat\cQ$
consist of a small number of test parameter samples, here chosen arbitrarily
as the first eight,
and compute
\begin{equation}
  \label{eq:K-est}
  K_{\rm est}:=\min_{\alpha\in\widehat\cQ_0}
  \frac{\|\clu_{h^\star}(\alpha)-\clz(\alpha)\|_{\Ua}}
       {\|\clu_{h^\star}(\alpha)-\clu_{h^\star/2}(\alpha)\|_{\Ua}}.
\end{equation}
The rationale is as follows. If we are at saturation, i.e., if there
is a $\rho<1$ with
$\|\clu-\clu_{h^\star/2}\|_{\Ua}\le\rho\,\|\clu-\clu_{h^\star}\|_{\Ua}$,
then the triangle inequality gives
$\|\clu-\clu_{h^\star}\|_{\Ua}\le C_{\rm sat}\,
\|\clu_{h^\star}-\clu_{h^\star/2}\|_{\Ua}$ with
$C_{\rm sat}:=(1-\rho)^{-1}$, and hence, for every
$\alpha\in\widehat\cQ_0$,
\begin{equation}
  \label{eq:K-est-perturbation}
  \bigg|\frac{\|\clu(\alpha)-\clz(\alpha)\|_{\Ua}}
             {\|\clu_{h^\star}(\alpha)-\clz(\alpha)\|_{\Ua}}-1\bigg|
  \le\delta:=\frac{C_{\rm sat}}{K_{\rm est}}.
\end{equation}
If~\eqref{eq:K-est-perturbation} held at every $\alpha\in\widehat\cQ$
and $\delta<1$, summing the squares would show that the ideal ratio
$\cl L_h(\clz;\widehat\cQ)/\cl E_h(\clu-\clz;\widehat\cQ)$ lies
between $(1+\delta)^{-2}\,\bar r_{\rm est}$ and
$(1-\delta)^{-2}\,\bar r_{\rm est}$.  Thus the larger the value of
$K_{\rm est}$, the smaller the effect of replacing $\clu$ by
$\clu_{h^\star}$. The constant $C_{\rm sat}$ is unknown, and the bound
is checked only on $\widehat\cQ_0$, so this is merely a plausibility
check. (With $K_{\rm est}\ge7.14$ throughout
Table~\ref{tab:posteriori-ratio}, the replacement could account for
the variation of $\bar r_{\rm est}$ across the boxes only if
$C_{\rm sat}$ were comparable to $K_{\rm est}$.)
The observed ratios $\bar r_{\rm est}$ in Table~\ref{tab:posteriori-ratio} vary
across the boxes by a factor of about $70$ for MDPG and $500$ for MFOSLS, although the contrast $\alpha_2/\alpha_1=2$ is the
same on all three. What changes is the coefficient scale, and with it the
weights $\kappa_1,\kappa_2$ and hence the norms in which loss and error are
measured. This shows that variational correctness does not
make the raw loss a universally calibrated error measure: its scale
depends on the parameter box and the norm weights.

\subsection{Network depth} %
\label{ssec:netcap}

We study how the test loss values change depending
on how many nonlinear layers are composed in the network, holding width
$N=128$ and rank $r=32$ fixed.
Figure~\ref{fig:section6-capacity} compares the depths $l\in\{1,3,5,7\}$ at
$h=0.1$, $k=1$, and $n_{\rm train}=4096$ through the cumulative minimum of
the test loss over checkpoints saved every $1000$ epochs,
\begin{equation}
  \label{eq:cmin}
  \mathrm{cmin}(e):=\min_{e'\le e}
  \cl L_h\big(\clz^{(e')};\widehat\cQ\big)^{1/2},
\end{equation}
where $\clz^{(e)}$ is the network after $e$ epochs. For MDPG with $l=1$,
$\mathrm{cmin}$ does not improve after epoch $3000$, and in both formulations
$l=1$ and $l=3$ end well above $l=5$ and $l=7$. Shallower networks are
therefore insufficient for this example at this training budget. For
$l\ge3$, $\mathrm{cmin}$ is still decreasing at $10^4$ epochs, so we cannot conclude that
training has converged there.
Continued to $2\times10^4$ epochs (dashed), the run with $l=5$ gains only
about another factor two, compared with nearly two orders of magnitude
between the first checkpoint and $10^4$ epochs. We therefore use $l=5$ and $2\times10^4$ epochs as a
practical compromise between accuracy and cost.

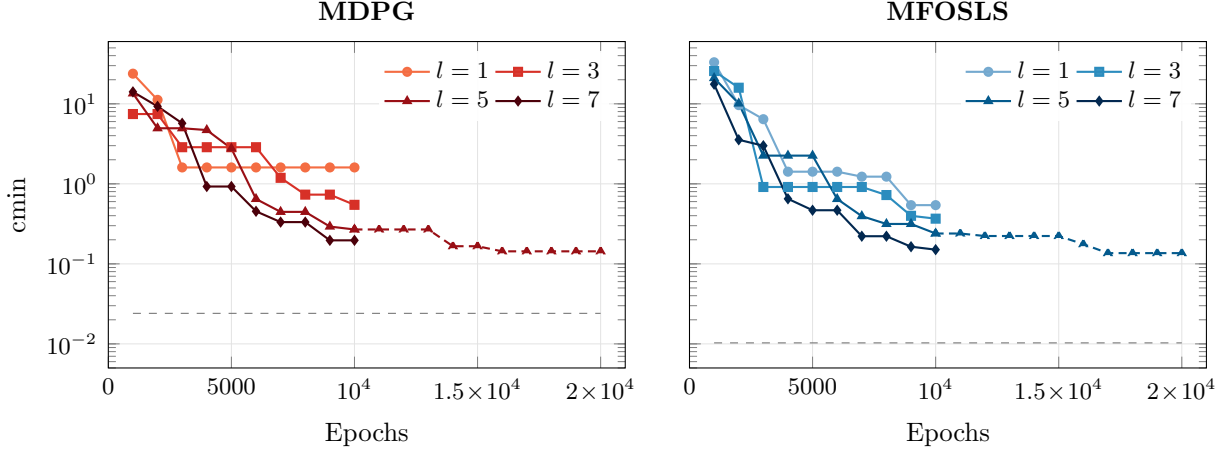
\begin{figure}[htbp]
  \centering
  \begin{tikzpicture}
    \begin{groupplot}[s6ablation, width=8.42cm,
        group style={group size=2 by 1, horizontal sep=0.85cm,
                     ylabels at=edge left, yticklabels at=edge left},
        xlabel={Epochs}, ylabel={$\mathrm{cmin}$},
        xmin=0, xmax=21000, scaled x ticks=false,
        xtick={0,5000,10000,15000,20000},
        xticklabels={$0$,$5000$,$10^4$,$1.5\!\times\!10^4$,$2\!\times\!10^4$},
        ymode=log, ymin=5e-3, ymax=6e1,
        legend style={at={(0.97,0.97)}, anchor=north east},
        legend columns=2]
      \nextgroupplot[title={MDPG}]
      \addplot[s6best, depthA, mark=*]         table[x=epoch, y=l1best] {data/section6_depth_mdpg.dat};
      \addplot[s6best, depthB, mark=square*]   table[x=epoch, y=l3best] {data/section6_depth_mdpg.dat};
      \addplot[s6best, depthC, mark=triangle*] table[x=epoch, y=l5best] {data/section6_depth_mdpg.dat};
      \addplot[s6best, depthD, mark=diamond*]  table[x=epoch, y=l7best] {data/section6_depth_mdpg.dat};
      \addplot[s6best, depthC, mark=triangle*, densely dashed, forget plot,
               restrict x to domain=10000:20000]
        table[x=epoch, y=n4096best] {data/section6_trainset_mdpg.dat};
      \addplot[black!55, dashed, no marks, forget plot]
        table[x=epoch, y=fem] {data/section6_trainset_mdpg.dat};
      \legend{$l=1$,$l=3$,$l=5$,$l=7$}
      \nextgroupplot[title={MFOSLS}]
      \addplot[s6best, depthAb, mark=*]         table[x=epoch, y=l1best] {data/section6_depth_mfosls.dat};
      \addplot[s6best, depthBb, mark=square*]   table[x=epoch, y=l3best] {data/section6_depth_mfosls.dat};
      \addplot[s6best, depthCb, mark=triangle*] table[x=epoch, y=l5best] {data/section6_depth_mfosls.dat};
      \addplot[s6best, depthDb, mark=diamond*]  table[x=epoch, y=l7best] {data/section6_depth_mfosls.dat};
      \addplot[s6best, depthCb, mark=triangle*, densely dashed, forget plot,
               restrict x to domain=10000:20000]
        table[x=epoch, y=n4096best] {data/section6_trainset_mfosls.dat};
      \addplot[black!55, dashed, no marks, forget plot]
        table[x=epoch, y=fem] {data/section6_trainset_mfosls.dat};
      \legend{$l=1$,$l=3$,$l=5$,$l=7$}
    \end{groupplot}
  \end{tikzpicture}
  \caption{The cumulative minimum~\eqref{eq:cmin} against the number of training
  epochs for depths $l\in\{1,3,5,7\}$, on $\cl Q_1$ with $h=0.1$, $k=1$,
  $n_{\rm train}=4096$, $N=128$, and $r=32$. The run with $l=5$ is continued to $2\times10^4$
  epochs (dashed part). The gray dashed line is the FEM value of $\cl L_h^{1/2}$
  in the same trial space.}
  \label{fig:section6-capacity}
\end{figure}

\subsection{Training and spatial accuracy}
\label{ssec:trainset-experiments}
\label{ssec:box-experiments}

We now fix $\cl Q_1$. Componentwise accuracy is measured against a common
high-fidelity reference: a conforming MFOSLS Newton solution with degree
$2$ and mesh size $0.0125$. For either network or FEM states, we report
\begin{equation}
  \label{eq:nn-reference-error}
  E_\zeta^{\,s}:=
  \left(\frac{\sum_{\alpha\in\widehat\cQ}
        \|\zeta_h^{\,s}(\alpha)-\zeta_{\rm ref}(\alpha)\|_{L_2}^2}
       {\sum_{\alpha\in\widehat\cQ}
        \|\zeta_{\rm ref}(\alpha)\|_{L_2}^2}\right)^{1/2},
  \qquad s\in\{{\rm NN},{\rm FEM}\}.
\end{equation}
For MDPG, $E_{\hat u}$ and $E_{\hat q_n}$ measure the conforming minimal
extensions of the traces in~\eqref{eq:mdpg-discrete}; for MFOSLS, $E_u$
and $E_q$ measure the conforming scalar and flux directly.

\begin{table}[htbp]
  \centering
  \footnotesize
  \setlength{\tabcolsep}{5pt}
  \caption{Effect of training-set size on $\cl Q_1$, with $h=0.1$,
  $k=1$, $l=5$, and $20\,000$ epochs. The upper block reports the loss and
  the errors~\eqref{eq:nn-reference-error} of the network; the row FEM
  reports those of the finite element solutions in the same trial space,
  which do not depend on training.}
  \label{tab:ntrain-errors}
  \begin{tabular}{r|ccc|ccc}
    \hline
    & \multicolumn{3}{c|}{MDPG} & \multicolumn{3}{c}{MFOSLS}\\
    $n_{\rm train}$ & $\cl L_h^{1/2}$ & $E_{\hat u}^{\rm NN}$ & $E_{\hat q_n}^{\rm NN}$
       & $\cl L_h^{1/2}$ & $E_u^{\rm NN}$ & $E_q^{\rm NN}$\\
    \hline
    128  & 58.9  & 0.134  & 8.00  & 33.3  & 0.104  & 5.43\\
    256  & 14.8  & 0.0607 & 4.16  & 15.0  & 0.0560 & 2.64\\
    512  & 5.27  & 0.0714 & 1.58  & 6.28  & 0.0409 & 0.991\\
    1024 & 1.62  & 0.0415 & 0.328 & 2.75  & 0.0361 & 0.269\\
    2048 & 0.555 & 0.0218 & 0.110 & 0.396 & 0.0184 & 0.0836\\
    4096 & 0.442 & 0.0482 & 0.212 & 0.195 & 0.0166 & 0.0614\\
    \hline
    & $\cl L_h^{1/2}$ & $E_{\hat u}^{\rm FEM}$ & $E_{\hat q_n}^{\rm FEM}$
       & $\cl L_h^{1/2}$ & $E_u^{\rm FEM}$ & $E_q^{\rm FEM}$\\
    \hline
    FEM  & 0.0241 & 0.00275 & 0.0106 & 0.0103 & 0.00072 & 0.0106\\
    \hline
  \end{tabular}
\end{table}

Table~\ref{tab:ntrain-errors} shows that,
for fixed network architecture,
increasing $n_{\rm train}$ from
$128$ to $4096$ substantially reduces the test loss and flux error in
both formulations. The component errors need not decrease monotonically:
the MDPG network errors increase from $2048$ to $4096$ samples despite the lower
loss.
This is most likely not an effect of the sample size but of two other
factors. First, the table reports the final
checkpoint, and late in training the loss fluctuates by up to a factor of
three between checkpoints: for $n_{\rm train}=4096$, the final MDPG value
$0.442$ lies well above the best value $0.143$ of the same run (dashed
curve in Figure~\ref{fig:section6-capacity}).
Second, the mean loss $\cl L_h(\clz;\widehat\cQ)$ is equivalent to the mean
squared error $\cl E_h(\clu-\clz;\widehat\cQ)$ in the weighted trial norm,
but it does not control each component error separately.
Even the largest training set leaves a substantial gap to the FEM
benchmark. Thus, the learning error
dominates the discretization error.  Since the loss decreases by more than
two orders of magnitude from $128$ to $4096$ samples, with gains at every
step, this gap can be attributed to the present weight and training budgets.
Further increasing the prediction accuracy would require larger networks,
more training data, and longer training: in line with the limited scope stated at
the beginning of this section, we do not pursue this here.
Because the batch size is fixed, a larger training set also means more
optimization steps per epoch, so the improvement reflects both more data
and more optimization.

\begin{table}[htbp]
  \centering
  \footnotesize
  \setlength{\tabcolsep}{5pt}
  \caption{Effect of degree on $\cl Q_1$ at $h=0.2$, with $l=5$,
  $n_{\rm train}=4096$, and $20\,000$ epochs. Each NN/FEM pair uses
  the same trial space and the same high-fidelity reference.}
  \label{tab:nn-errors}
  \begin{tabular}{cc|ccc|ccc}
    \hline
    & & \multicolumn{3}{c|}{MDPG} & \multicolumn{3}{c}{MFOSLS}\\
    $k$ & & $\cl L_h^{1/2}$ & $E_{\hat u}$ & $E_{\hat q_n}$
           & $\cl L_h^{1/2}$ & $E_u$ & $E_q$\\
    \hline
    0 & NN  & 1.65   & 0.576   & 0.228  & 0.254   & 0.101   & 0.219\\
      & FEM & 1.65   & 0.737   & 0.234  & 0.245   & 0.0981  & 0.214\\
    1 & NN  & 0.103  & 0.0276  & 0.0417 & 0.0554  & 0.0153  & 0.0398\\
      & FEM & 0.0905 & 0.0315  & 0.0351 & 0.0347  & 0.00485 & 0.0350\\
    2 & NN  & 0.0428 & 0.0158  & 0.0332 & 0.0365  & 0.0161  & 0.0252\\
      & FEM & 0.00761 & 0.00438 & 0.0237 & 0.00495 & 0.00040 & 0.00517\\
    \hline
  \end{tabular}
\end{table}

Table~\ref{tab:nn-errors} restricts the spatial comparison to $h=0.2$
and varies $k$. Raising the degree from $0$ to $1$ sharply improves
the network predictions in both formulations. The error of a network has
two contributions: the discretization error, measured by the FEM rows, and
the error of learning the dependence on the parameter. From $k=1$ to $k=2$,
the FEM losses fall much more rapidly than those of the networks, so at
$k=2$ the learning error dominates, as already observed in
Table~\ref{tab:ntrain-errors}.
Enriching the trial space reduces only
the discretization error. Hence, to benefit from a richer trial space, the
learning error must be reduced as well, through more network capacity,
training data, and optimization effort.

The lowest-order MDPG scalar error ($E_{\hat u}$) in
Table~\ref{tab:nn-errors} is particularly striking: it is $0.576$ for
the network and $0.737$ for FEM, compared with $0.101$ and $0.0981$
for MFOSLS ($E_u$). The flux errors, in contrast, are similar. Because
the poor scalar accuracy is already present in the MDPG finite element
solution, it cannot be attributed solely to neural-network training.
The discrepancy is greatly reduced at $k=1$. These observations
motivate using degree $k\ge 1$ for the present example and
suggest that further studies are needed to examine whether the
lowest-order case suffers from locking phenomena.

\subsection{A sixteen-parameter example}
\label{ssec:mixed-box-experiments}

The experiments so far use four parameters and contrast $2$. This last
experiment tests whether the same architecture and optimizer settings carry
over, without retuning, to a higher-dimensional parameter space with a larger
contrast and a more intricate coefficient pattern. We divide $\om$ into a
$4\times4$ array of subdomains and sample uniformly from
\begin{equation}
  \label{eq:num-mixed-boxes}
  \cl Q_{\rm mix}=\prod_{i=1}^{16}J_{n(i)}.
\end{equation}
The fixed assignment places four subdomains in $J_1$, seven in $J_2$,
and five in $J_3$. Hence $\alpha_1=1$, $\alpha_2=6$, and the weights
from theory are $\mu=82$, $\kappa_1=9$, and $\kappa_2=82/3$.
Both networks use $h=0.1$, $k=1$, $n_{\rm train}=4096$, and $15\,000$ training epochs.

\begin{figure}%
  \centering
  \begingroup
  \setlength{\tabcolsep}{3pt}
  \begin{tabular}{@{}cc@{\hspace{24pt}}cc@{}}
    \includegraphics[width=.22\textwidth]{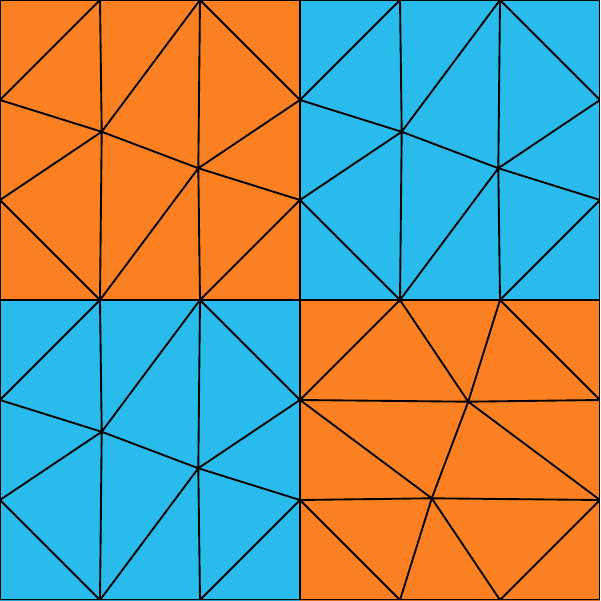}
    & \includegraphics[height=.22\textwidth]{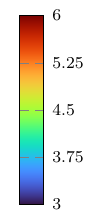}
    & \includegraphics[width=.22\textwidth]{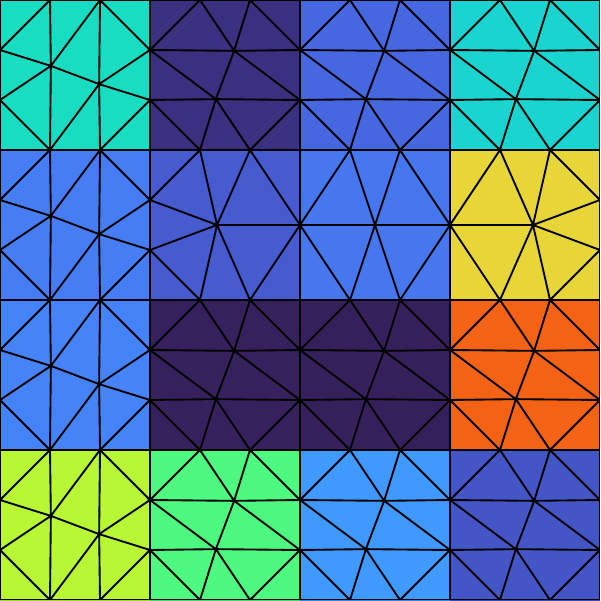}
    & \includegraphics[height=.22\textwidth]{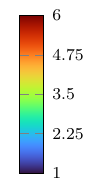}\\
    \multicolumn{2}{c}{\footnotesize $2\times2$ partition, $\cl Q_1$}
    & \multicolumn{2}{c}{\footnotesize $4\times4$ partition, $\cl Q_{\rm mix}$}
  \end{tabular}
  \par\smallskip
  \edef\panw{\the\dimexpr.215\textwidth\relax}
  \newcommand{\solpanel}[1]{%
    \includegraphics[width=\panw]{figures/section6_mixed_box_ref/panels/#1.pdf}}
  \newcommand{\solcbar}[1]{%
    \includegraphics[height=\panw]{figures/section6_mixed_box_ref/colorbars/#1.pdf}}
  \begin{tabular}{rccc@{\hspace{3pt}}c}
    & \footnotesize reference & \footnotesize MDPG & \footnotesize MFOSLS &\\
    $u$ & \solpanel{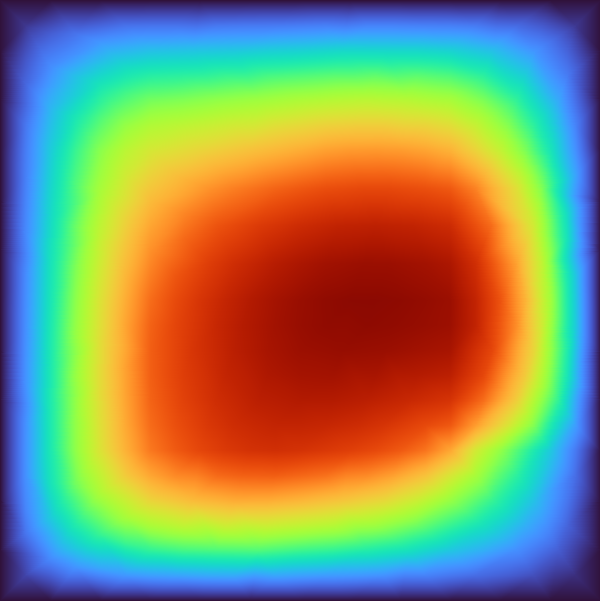} & \solpanel{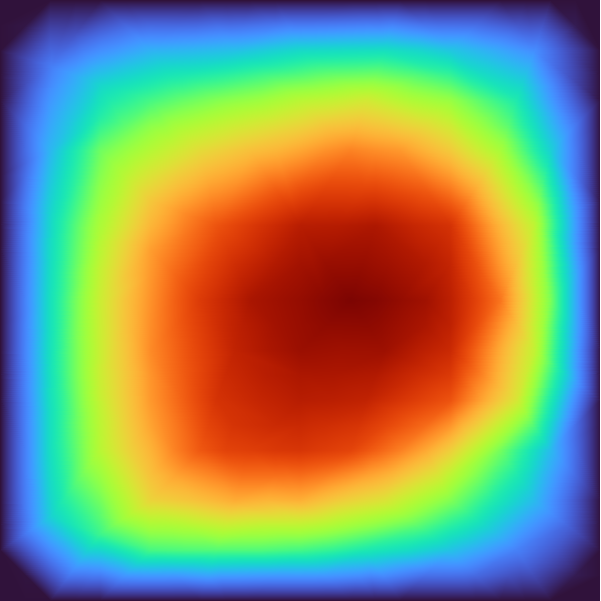}
        & \solpanel{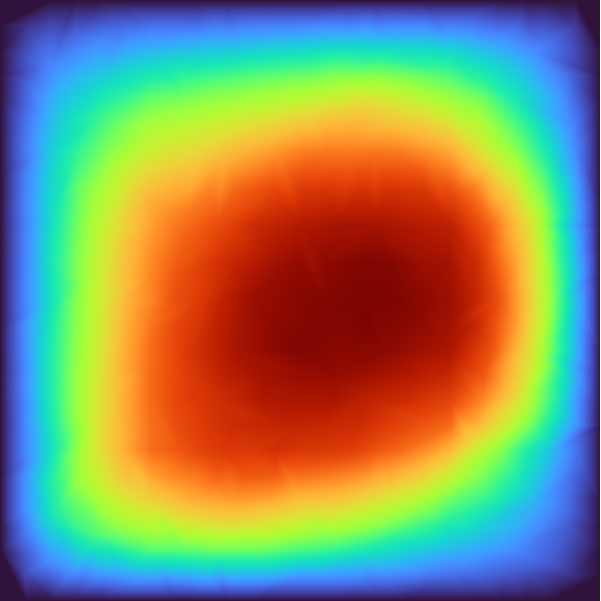} & \solcbar{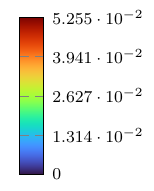}\\
    $|q|$ & \solpanel{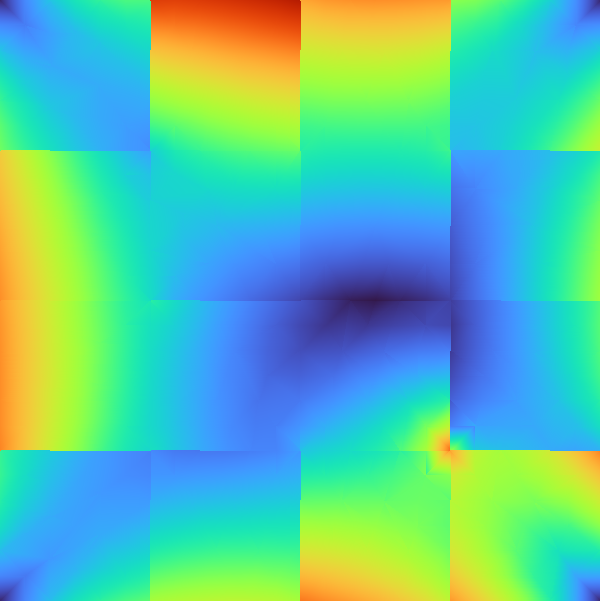} & \solpanel{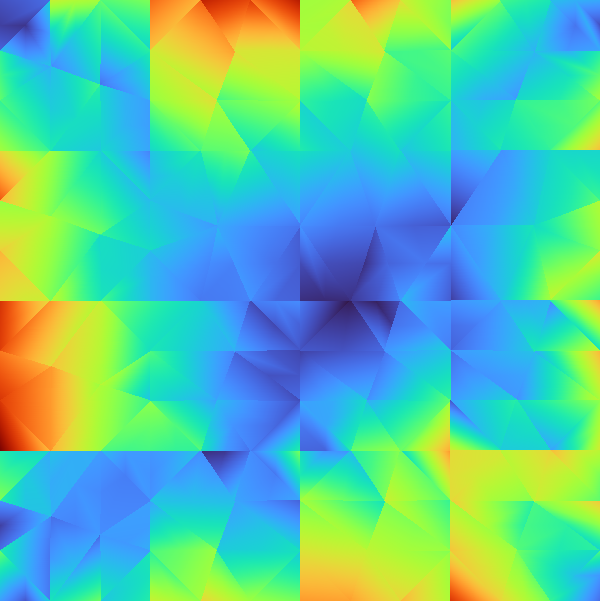}
          & \solpanel{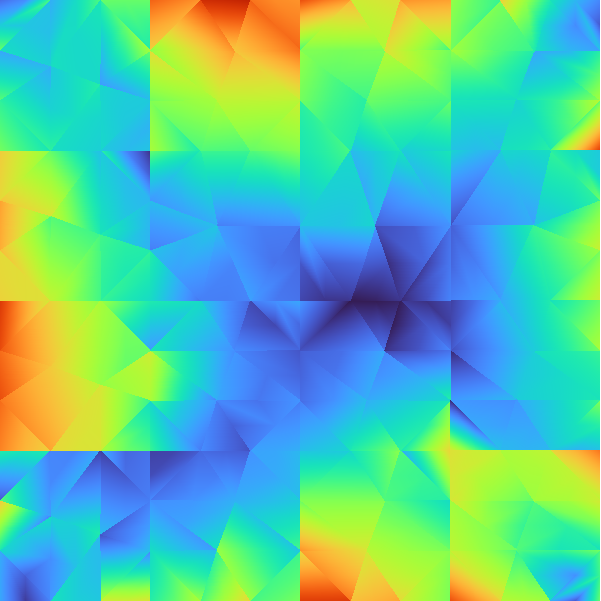} & \solcbar{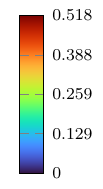}\\
  \end{tabular}
  \endgroup
  \caption{Top: sample coefficients on the $2\times2$ partition used in
  \S\S\ref{ssec:6.1}--\ref{ssec:trainset-experiments}, drawn from $\cl Q_1$
  on the mesh with $h=0.2$ (left), and on the $4\times4$ partition, drawn
  from $\cl Q_{\rm mix}$ on the mesh with $h=0.1$ (right). Each color scale
  covers the range of the respective box. Bottom: reference and network
  solutions for the $\cl Q_{\rm mix}$ sample. The reference
  is a conforming Newton solution with $k=1$ and $h=0.05$; the networks
  use $k=1$ and $h=0.1$. MDPG displays conforming trace extensions.
  Each solution row uses a common color scale.}
  \label{fig:mixed-box-solutions}
\end{figure}

Although a single sample does not quantify accuracy, for illustration, 
Figure~\ref{fig:mixed-box-solutions} shows that both MDPG and MFOSLS neural network predictions reproduce
the main scalar and flux features for the displayed sample. In particular,
the reference flux magnitude changes abruptly across the subdomain
interfaces in a pattern set by all sixteen coefficient values, and both
networks reproduce the main features of this pattern. Thus, with four times
as many parameters and three times the contrast, the architecture and
optimizer settings of the preceding experiments still yield predictions
that resolve the coefficient structure.

\section{Conclusions}

We have introduced and analyzed two loss functions for a
class of nonlinear diffusion models that are globally variationally correct, i.e.,
trial norm errors are proportional to corresponding residuals measured in the dual test-norm.
In one formulation (MFOSLS) the test spaces are of $L_2$-type so that $L_2$-residuals 
are computable. In the other formulation (MDPG) the test space is not self-dual and we resort
to DPG concepts for generating computable loss expressions. In both cases variational 
correctness hinges on {\em monotonicity principles} which account for global stability.
However, the price for this latter asset is to work with certain parameter-dependent trial norms.
In the learning context this renders the error--loss relationship parameter-dependent. In order
to shed further light on the quantitative interplay between the error--loss relation on the one hand
and, on the other hand, the role of various hyperparameters of our hybrid FE-NN
hypothesis classes, we have designed
a few simple numerical experiments. They confirm the parameter dependence of the
error--loss relation, show that sufficiently deep networks are needed, and show that the
error of the trained networks is dominated by the learning error rather than by the
finite element discretization. Overall, MFOSLS attains a slightly better accuracy in most
experiments. Moreover, at lowest order, MDPG exhibits a large scalar error that is already
present in its finite element solution and does not appear with MFOSLS; whether this is a
locking phenomenon remains to be examined. This contrasts the situation of
linear diffusion models \cite{CDG} where the DPG approach turned out to be significantly 
more robust with regard to parameter variation, largely due to the fact that the DPG formulation
allows one to push the test-norms close to so-called {\em optimal test norms}.
Such renormalization ideas for the nonlinear context will be taken up in forthcoming work. 

\section*{Statements and Declarations}

\noindent\textbf{Funding.} This work was supported in part by NSF Grants DMS-2245077, DMS-2513112, DMS-2245097, DMS-2038080,   and by the German Research Foundation grant CRC 1481.

\smallskip\noindent\textbf{Competing interests.} The authors declare no competing interests.

\smallskip\noindent\textbf{Data availability.} The data reported in Section~\ref{sec:numerics} were generated by the authors and are available from the corresponding author on reasonable request.

\end{document}